\documentclass[11pt, twoside]{article}
\usepackage{adjustbox}
\usepackage{amsfonts}
\usepackage{amsmath, amsthm, amscd, amsfonts, xpatch}
\usepackage{times}
\usepackage{atbegshi,picture}
\usepackage{enumitem,kantlipsum}
\setlist[enumerate]{itemsep=0mm}
\usepackage[english]{babel}
\usepackage[T1]{fontenc}
\usepackage[utf8]{inputenc}
\usepackage{bm}
\usepackage{ifxetex}
\usepackage{pb-diagram}
\usepackage{mathtools}
\usepackage{mathrsfs}
\usepackage{latexsym}
\usepackage{amssymb}
\usepackage{caption}
\usepackage{graphicx,float}
\usepackage{tikz}
\usepackage{pst-node}

\usepackage[a4paper,inner=2cm,outer=2cm,top=2cm,bottom=2cm,footskip=1cm]{geometry}

\def\titlerunning#1{\gdef\titrun{#1}}
\makeatletter
\def\author#1{\gdef\autrun{\def\and{\unskip, }#1}\gdef\@author{#1}}

\makeatother

\def\keywords#1{\par\medskip
	\noindent\textbf{Keywords.} #1}
\def\MSC#1{\par\medskip
	\noindent\textbf{MSC.} #1}
\usepackage{lineno}
\usepackage{csquotes}

\definecolor{persianblue}{rgb}{0.11, 0.22, 0.73}
\definecolor{persiangreen}{rgb}{0.0, 0.65, 0.58}

\usepackage{url}

\usepackage{nameref}
\usepackage[hidelinks]{hyperref}
\hypersetup{
	colorlinks   = arscenic, %Colours links instead of ugly boxes
	urlcolor     = arscenic, %Colour for external hyperlinks
	linkcolor    = arscenic, %Colour of internal links
	citecolor   = arscenic %Colour of citations
}
\usepackage[capitalize]{cleveref}
\Crefname{theorem}{Theorem}{Theorems}
\Crefname{lemma}{Lemma}{Lemmas}
\crefname{claim}{Claim}{Claims}

\usepackage[defernumbers=true,maxbibnames=6,backend=biber,style=numeric,giveninits=true,sortcites=true,doi=false,isbn=false,url=false,safeinputenc]{biblatex}
\usepackage[nottoc,numbib]{tocbibind} 
\usepackage{tikz-cd}

\newcommand{\ORD}{{\rm ORD}}
\newtheorem{theorem}{Theorem}[section]
\newtheorem*{theorem*}{Theorem}
\newtheorem*{mainthm*}{Main Theorem}

\newtheorem{lemma}[theorem]{Lemma}

\newtheorem{proposition}[theorem]{Proposition}
\newtheorem{corollary}[theorem]{Corollary}
\newtheorem{problem}[theorem]{Problem}
\newtheorem{problem*}{Problem}[]

\newtheorem{fact}[theorem]{Fact}
\newtheorem{definition}[theorem]{Definition}

\newtheorem{question}[theorem]{Question}
\newtheorem{notation}[theorem]{Notation}

\numberwithin{equation}{section}
\newtheorem{convention}[theorem]{Convention}

\newcommand{\cof}{{\rm cof}}

\newcommand{\rest}{\! \upharpoonright \!}

\newcommand\gesm{\emph{gesm}}
\newcommand\gesms{\emph{gesms}}

\newtheorem{imp-remark}[theorem]{\textbf{Important Remark}}
\numberwithin{equation}{section}

\theoremstyle{remark}

\makeatletter
\let\qed@empty\openbox % <--- change here, if desired
\def\@begintheorem#1#2[#3]{%
	\deferred@thm@head{%
		\the\thm@headfont\thm@indent
		\@ifempty{#1}
		{\let\thmname\@gobble}
		{\let\thmname\@iden}%
		\@ifempty{#2}
		{\let\thmnumber\@gobble\global\let\qed@current\qed@empty}
		{\let\thmnumber\@iden\xdef\qed@current{#2}}%
		\@ifempty{#3}
		{\let\thmnote\@gobble}
		{\let\thmnote\@iden}%
		\thm@swap\swappedhead
		\thmhead{#1}{#2}{#3}%
		\the\thm@headpunct\thmheadnl\hskip\thm@headsep
	}\ignorespaces
}
\renewcommand{\qedsymbol}{%
	\ifx\qed@thiscurrent\qed@empty
	\qed@empty
	\else
	\fbox{\scriptsize\qed@thiscurrent}%
	\fi
}
\renewcommand{\proofname}{%
	Proof%
	\ifx\qed@thiscurrent\qed@empty
	\else
	\ of \qed@thiscurrent
	\fi
}
\xpretocmd{\proof}{\let\qed@thiscurrent\qed@current}{}{}
\newenvironment{proof*}[1]
{\def\qed@thiscurrent{\ref{#1}}\proof}
{\endproof}
\makeatother

\makeatletter
\let\blx@rerun@biber\relax
\makeatother
\clearpage

\begin{document}
	
	%%%%% To ease editing, add:
	
	\baselineskip=17pt
	
	%%%%%%%%%%%%%%%%

	%% In the running head, give an abbreviation of the title. 
	\titlerunning{A Road To Compactness Through Guessing Models}
	
	\title{A Road To Compactness Through\\
		Guessing Models}
	
	\author{Rahman Mohammadpour\bigskip\\
		\small{\textit{Institut für Diskrete Mathematik und Geometrie,}}\\
		\small{\textit{Technische Universität Wien,}}\\
		\small{\textit{Wiedner Hauptstrasse 8-10/104, 1040 Vienna, Austria}}\\
			\vspace{1cm}	\small{\href{mailto:rahmanmohammadpour@gmail.com}{rahmanmohammadpour@gmail.com}}}
	\date{}
	
	\maketitle
	
	%\address{Rahman Mohammadpour: Institut für Diskrete Mathematik und Geometrie, TU Wien,1040 Vienna, Austria.; \email{rahmanmohammadpour@gmail.com}}

	%\subjclass{03E35}
	
	%%%%%%%%
	
	\begin{abstract} 
		
		The compactness phenomenon is one of the featured aspects of structuralism in mathematics. In simple and broad words, a compactness property holds in a structure if a related property is satisfied by sufficiently many substructures of that structure. With this phenomenon and its twin sibling "reflection", modern set theory has settled many mathematical statements left undecided by the conventionally accepted formalism of mathematics, $\rm ZFC$. An ongoing broad  research programme investigates whether a concept of compactness can universe-widely emerge without running into contradictions. 
		
		These notes are a survey of notions of guessing models whose existence provides intriguing compactness phenomena. Most of the results in the manuscript are well-known. We shall reformulate, generalise and expand some of them. We also present some known applications of guessing models and state some open problems.
		
		\keywords{Compactness, Guessing Model, Proper Forcing Axiom (PFA),  Reflection}
		\MSC{03E05, 03E35, 03E57, 03E65}
		
	\end{abstract}
	
	%\blfootnote{}
	
	%\tableofcontents

	\section{Introduction}
I left my home the day before my 17th birthday anniversary in 7104 to become an ultraset theorist. I was young and, like many youths, so enthusiastic about devoting my entire life to a meaningful answer to the ancient question of the Continuum Hypothesis. I so well remember the moment I was about to step into the history class. It was my first day, how can I forget it?The classroom was approachable from a corridor on the second floor surrounded by numerous old but familiar portraits of foundationers hanging on the walls. Although, I was so impressed by them not notice the class had started. Closer to the door, was a paled shadow dancing over the portraits drowned in the soft edges of the light. Running through them, they caressed my face; what a pleasant moment remained unique forever. I was late, of course! So late to a pretty long classroom gone deep into perfect silence. It was breathtaking. The professor or, hmm, better to say the silhouette in the tender light of the dawning sun, momentarily paused by my presence at the door frame. Time halted in my head. I sweated and did feel deep inside that she was staring at me in anger.  I moved inside by her nod while my heavy heart was still swinging. Luckily enough, I was to find a sit on the second row. That silhouette now I could recognise was an expert in collapse magic. She continued her lesson: ``A marvellous aspect of the sierra of large cardinals is that the power of many vigorous dragons conceal behind the climax of inaccessibility. They, though, appeared from time to time, would rapidly fly up to the scary condensed clouds of the death world. Let me tell you an astonishing epic that narrates a significant two-person adventure by Alfred the Leader and Paul the King in the history of sets. They revealed the weakly compact dragons hiding behind the branches of Aronszajn trees in the vastly frozen forests on the summit of inaccessibility.   That persistently stimulated the young generations to pursue their dreams and seek familiar-looking visages of dragons around the world of large cardinals. Alas, it was not always a victorious path; many failed, many unreturned, and many lost. That's sad, I do know! An astonishing but perilous approach that wilful people would nevertheless take was to burn a feather of Simurgh to assist them in hunting a dragon and bringing it down to an accessible summit. It could be very well-respected and praised. These are the celebrated stories. Once upon a time, when our world was a greener and larger place to live, two renowned and experienced explorers independently decided to discover the unattainable levels of strong compactness and supercompactness.   They were Thomas (`the author of the Green Book, excuse me?', I interrupted the professor, but she was not unhappy this time) and Menachem, the logician. These self-sacrificing people unhesitatingly took long foot journeys on the cruel rocks of enormous mountains while the furious wind heartlessly moved them back under the order of dragons. Destiny has no place there! It's the kingdom of dragonish chaos. Of course, so committed they were to identify those dragons near the enormous branches of some pseudo-tree entities living there, again on unimaginably high inaccessible peaks. They accomplished! We will discuss those entities later, but be aware that they were never as tame as Aronszajn trees. Yet, somewhat later, Christoph, a junior from Deutschland, pulled those dragons down to accessible mountains. An accomplishment that finally resulted in a phenomenal collaborative journey with Matteo di Torino in the discovery of a new extraordinary creature. A unique type of baby dragon they named a guessing model!''  

	\section{Basics and Notation}
	For a cardinal $\theta$, $H_{\theta}=(H_\theta,\in)$ denotes the collection  
	of sets with hereditary size less than $\theta$.
	For a set $x$, we denote  the power-set of $x$ by  $\mathcal P(x)$, and
	for a cardinal $\kappa$, we let $\mathcal P_\kappa(x)\coloneqq\{y\subseteq x: |y|<\kappa\}$. A set $\mathcal S\subseteq\mathcal P_\kappa(x)$ is called \emph{stationary} if for every function $F:\mathcal P_{\omega}(x)\rightarrow\mathcal P_\kappa(x)$, there is $A\in\mathcal S$ such that $A$ is closed under $F$, i.e. $F(a)\subseteq A$, for every $a\in \mathcal P_\omega(A)$.

	By a \textbf{model} $M$, we mean a set or class such that $(M,\in)$ satisfies a sufficiently strong fragment of $\rm ZFC$. 
	We use $\overline{M}$ for the transitive collapse of $(M,\in)$, where the transitive collapse map is often denoted by $\pi_M$. A set $x$ is \textbf{bounded} in $M$, if there is $y\in M$ with $x\subseteq y$.  By a \textbf{powerful model} $\mathcal H$, we mean that $\mathcal H$ is a transitive  model which contains every set bounded  in $\mathcal H$.
	Note that  $H_\theta$ is a powerful model.

	\begin{definition}
		Let $M$ be a model.
		\begin{enumerate}
			\item 
			$M$ is \textbf{$\bm \delta$-close} to an ordinal $\gamma$ if
			$\cof(\sup(M\cap\gamma))=\delta$.
			\item $M$ has \textbf{bounded uniform cofinality $\bm\delta$} if for every $\gamma\in M\cap \ORD$, if $M\models ``\cof(\gamma)\geq\delta"$, then $M$ is $\delta$-close to $\gamma$, and 
			\item $M$ has \textbf{uniform cofinality $\bm \delta$}, if, moreover, $\cof(\sup(M\cap\ORD))=\delta$.
			
		\end{enumerate}
		\begin{definition}
			Assume $M\subseteq N$ are models. Let $\delta\in M$ be an $N$-cardinal. 
			The pair $(M,N)$ has the \textbf{$\bm \delta$-covering property}  if for every  set $x\in N$ with $|x|^N\leq\delta$  which is bounded in $M$, there is $x'\in M$ with $|x'|^M\leq\delta$ such that $x\cap M\subseteq x'$.
		\end{definition}

	\end{definition}
	\begin{notation}
		For a set $X$, we  let $\kappa_X$ be the least ordinal $\alpha\in X$ such that $\alpha\nsubseteq X$, we leave $\kappa_X$  undefined if such an $\alpha$ does not exist.
	\end{notation}
	
	\section{Guessing Models}

	A guessing (elementary) submodel is, roughly speaking,  a set model in the universe which, to some extent and at some cost, is correct about the powerset operation.
	This paper is devoted to the compactness of the following intelligible property. 
	
	\begin{definition}
		Let $X$ be a set. 
		A set $x$ is said to be \textbf{guessed} in $X$, if there exists $x^*\in X$, such that
		$x^*\cap X=x\cap X$.
	\end{definition}
	
	Obviously, every set inside $X$ is guessed in $X$, and  that a set $x$ is guessed in $X$ if and only if $x\cap X$ is guessed in $X$. 
	\subsection{Equivalent definitions}
	A set $x$ is \textbf{$\bm \delta$-approximated in} a model $M$, for an $M$-cardinal $\delta$, if
	$x\cap a\in M$, for every $a\in M$ with $|a|^M<\delta$.

	\begin{definition}[Viale \cite{V2012}, Viale--Wei\ss~\cite{VW2011}]
		Suppose $M\subseteq N$ are  models.
		Assume that $\delta\in M$ is a cardinal with $\delta\leq\kappa_M$.
		We say $M$  has the \textbf{$\bm\delta$-guessing property in} $N$ if  every $x\in N$  bounded and  $\delta$-approximated in $M$ is guessed in $M$.
	\end{definition}
	
	Working in a model $N$,
	observe that if $\delta\leq\delta'$, then every $\delta$-guessing model  in $N$ is also $\delta'$-guessing in $N$. For the approximation side, if $x\subseteq X\in M$, then it is enough to show that for every $a\in\mathcal P^N(X)\cap M$ with $|a|^M<\delta$,
	$a\cap x\in M$ holds.
	There is a less useful reformulation of the above definition  which points more compactness out, that is  $x$ is guessed in $M$ if and only if for every $a\in M$ of size less than $\delta$, $x\cap a$ is guessed in $M$. This trivial reformulation follows from the fact that if
	$x \subseteq M$  is guessed in $M$ with  $|x|\in M\cap \kappa_M$, then $x\in M$!
	
	Observe that if $M$ is transitive, then $x$ is guessed in $M$ if and only if $x\in M$.
	In fact, one can now reformulate Hamkins' approximation property in terms of guessing property.
	
	\begin{definition}[Hamkins \cite{Hamkins2003}]
		Let $(M,N)$ be a pair of transitive models with $M\subseteq N$.
		Let $\delta\in M$ be a regular cardinal in $N$.
		Then $(M,N)$ has the $\delta$-approximation property if and only if $M$ is $\delta$-guessing in $N$.
	\end{definition}

	\begin{definition}
		We say that $M$ is \textbf{$\bm \delta$-guessing}  or has \textbf{the $\bm\delta$-guessing property} if and only if
		$M$ has the $\delta$-guessing property in $V$.
	\end{definition}
	
	The following is well-known.
	
	\begin{lemma}\label{transitivity under covering}
		Assume that $M_i$ has  the $\delta$-guessing property in $M_{i+1}$, for $i=0,1$.
		Suppose that $(M_0,M_1)$ has the $\delta$-covering property, then $M_0$ has the $\delta$-guessing property in $M_2$. 
	\end{lemma}
	\begin{proof}
		Assume that $x\in M_2$ is bounded and $\delta$-approximated in $M_0$.
		We first show that $x$ is $\delta$-approximated in $M_1$. Suppose $a\in M_1$ is of $M_1$-cardinality less than $\delta$. By the $\delta$-covering property, there is $b\in M_0$ with $a\subseteq b$ and $|b|^{M_0}<\delta$. Now $b\cap x\in M_0\subseteq M_1$, and hence 
		\[
		a\cap x=a\cap (b\cap x)\in M_1.
		\]
		Since $M_1$ is $\delta$-guessing in $M_2$, there is $x_1\in M_1$ with $x_1\cap M_1=x\cap M_2$. Observe that $x_1$ is bounded in $M_0$. It is enough to show that $x_1$ is $\delta$-approximated in $M_0$. Thus pick any $a\in M_0$ with $|a|^{M_0}<\delta$.
		Since $\delta\leq\kappa_{M_0}$,
		We have 
		\[
		a\cap x_1=a\cap x_2\in M_0.
		\]
	\end{proof}
	\begin{corollary}
		$M$ is a $\delta$-guessing if and only if it is $\delta$-guessing in some powerful model.
	\end{corollary}
	\begin{proof}[\nopunct]
	\end{proof}
	
	\begin{definition}
		Let $M$ be a model with a well-defined $\kappa_M$. Assume that $\delta\leq\kappa_M$ is a regular cardinal in $M$.
		We say $M$ is a \textbf{$\bm\delta$-guessing elementary submodel ($\bm\delta$-\gesm)} if $M$ is an elementary submodel of a powerful model $\mathcal H$, and that $M$ has the $\delta$-guessing property.
	\end{definition}
	Notice that if $\mathcal H'\prec \mathcal H$ are powerful models, and $M\prec \mathcal H$   is a $\delta$-\gesm~with $\delta\in\mathcal H'\in M$, then $M\cap\mathcal H'$ is a $\delta$-\gesm.

	One can also reformulate guessing models  in terms of functions, we leave the proof to the interested reader.
	However, we may use it without mentioning.
	
	\begin{proposition}\label{functional-def}
		Suppose $\mathcal H$ is a powerful model. Then a model $M\prec H$ is a $\delta$-\gesm~  if and only if for every  uncountable cardinal $\gamma\in M$ and every $f:\gamma\rightarrow 2$, if for all $x\in M\cap\mathcal P_{\delta}(\gamma)$, $f\rest_x\in M$, then there is a function $f^*\in M$ with ${\rm dom}(f^*)=\gamma$ such that $f^*\rest_M=f\rest_M$.
	\end{proposition}
	\begin{proof}[\nopunct]
	\end{proof}
	There is still a rather useful reformulation of a guessing model. 
	
	\begin{proposition}[Cox--Krueger \cite{CK-Quotient}]\label{CK-formulation}
		Assume that $M$ is an elementary submodel of a powerful model. Then $M$ is a $\delta$-\gesm~
		if and only if $\overline{M}$ is $\pi_M(\delta)$-guessing.
	\end{proposition}
	\begin{proof}[\nopunct]
	\end{proof}

	The definition of a \gesm~demonstrates that though $M$ may not be correct about the power-set, it still does not compute it as badly as we can imagine. As we shall see, if $x$ is a set with some additional structure and is guessed in $M$, say
	$x\cap M=x^*\cap M$ for some $x^*\in M$, then $x^*$ can be roughly regarded as a structure similar to that of $x$, while containing $x\cap M$ as a substructure!
	
	\subsection{The geometry of \gesms}
	
	What does a \gesm~look like?  In this subsection, we intend to try to find a reasonable answer to this question. In other words, how does being a \gesm~impact the structure itself.

	Intuitively, it would not seem easy to get more sets guessed with a small amount of approximation.
	To look at an extreme case, let us consider  $\delta=0$.
	This means that if $x\subseteq M$ is bounded in $M$, then there is $x^*\in M$ such that $x^*\cap M=x$. This itself implies that $M$ is really enormous. One can  find such a model $M$, that is if $\overline{M}=V_{\bar\gamma}$, for some limit ordinal $\bar{\gamma}$.
	
	\begin{theorem}[Viale \cite{V2012}]
		$M$ is a $0$-\gesm~if and only if for some limit ordinal $\bar{\gamma}$, $\overline{M}=V_{\bar{\gamma}}$.
	\end{theorem}
	\begin{proof}[\nopunct]
	\end{proof}

	Le us not forget that we are interested in elementary submodels of $H_\theta$ with the guessing property. Alas, finding many models of the above form and of size less than a given uncountable regular cardinal roughly needs a supercompact cardinal.
	This was proved by Viale  \cite{V2012} using the  the above theorem and a characterization of supercompactness due to Magidor  \cite{MA1971}. 
	Viale's paper \cite{V2012} contains   characterizations of other large cardinals in terms of $0$-\gesms. Since our focus in this paper is  on accessible cardinals, we encourage the interested reader to consult \cite{V2012} for more results of above sort.
	
	\begin{theorem}
		Suppose  $M$ is a $\delta$-\gesm. Let  $\theta\in M$ be a cardinal with ${\rm cof}(\theta)\geq\delta$. Assume that $M$ is closed under sequences of length ${<}{\rm cof}({\rm sup}(M\cap\theta))$ which are bounded in $M$.
		Then $ |M|\geq 2^{{\rm cof}({\rm sup}(M\cap\theta))}$.
	\end{theorem}
	\begin{proof}
		Let $\lambda={\rm cof}({\rm sup}(M\cap\theta))$, and let
		$c=\{\theta_\xi:\xi<\lambda\}\subseteq M\cap\theta$  be increasing and cofinal in ${\rm sup}(M\cap\theta)$.
		For every $z\subseteq\lambda$, let $c[z]=\{\theta_\xi:\xi\in z\}$.
		Note that $c[z]$ is bounded in $M$, as witnessed by $\theta\in M$. Note that also $c[z]\subseteq M$.
		We shall show that $c[z]$ is $\delta$-approximated in $M$.
		For every  subset $a\in\mathcal P_\delta(\theta)\cap M$, $a\cap c[z]$ is a bounded sequence of length less than $\lambda$, since otherwise $\sup(a)=\sup(M\cap\theta)\in M$, and hence 
		by elementarity, the cofinality of $\theta$ is less than $\delta$, which is a contradiction. Thus, by the assumption on the closure property, $a\cap c[z]\in M$.  Thus $c[z]$ is $\delta$-approximated in $M$.
		Therefore, there is $c^*[z]\in M$ with $c^*[z]\cap M=c[z]\cap M=c[z]$.
		Observe that, by elementarity, $c^*[z]$ is unique.
		It is easily seen, since $c$ is one-to-one, that if $z\neq z'$, then $c^*[z]\neq c^*[z']$, and hence
		the function $z\rightarrow c^*[z]$ is an injection from $\mathcal P(\lambda)$ into $M$. 
	\end{proof}
	
	Note that for every $n\in\omega$, being  $n$-guessing is equivalent to being $\omega$-guessing. Thus every $\delta$-\gesm~is uncountable.

	\begin{corollary}[Cox--Krueger \cite{CK2017}]
		If $M$ is an $\omega_1$-\gesm~of size ${<}2^{\aleph_0}$, then $\omega_1\subseteq M$, and for every   cardinal  $\theta\in M$ of uncountable cofinality, ${\rm cof}(\sup(M\cap\theta))$ is uncountable. In particular, if $M$ is of size $\omega_1$, then  ${\rm cof}(M\cap\theta)=\omega_1$ and $M\cap\omega_2\in \omega_2$. 
	\end{corollary}
	\begin{proof}[\nopunct]
	\end{proof}
	
	It is straightforward to generalise the  above corollary to $\delta\geq\omega_1$, when $M$ is sufficiently closed: 
	If $M$ is a ${<}\delta$-closed $\delta^+$-\gesm~size $\delta^+$, then $M\cap\delta^{++}\in\delta^{++}$ and for every cardinal $\theta\in M$ with ${\rm cof}(\theta)>\delta$, ${\rm cof}({\rm sup}(M\cap\theta))=\delta^{+}$. In particular, $\delta^+\subseteq M$.

	\begin{proposition}[Viale \cite{V2012}]\label{0-g}
		Suppose that $M$ is a $\delta$-\gesm. If $2^{<\delta}< \kappa_M$, then $M$ is a $0$-\gesm, and thus both $\sup(M\cap\kappa_M)$ and $\kappa_M$ are inaccessible.
	\end{proposition}
	\begin{proof}
		If $2^{<\delta}<\kappa_M$, then $2^{<\delta}\subseteq M$.
		It is enough to  show that every set $x$ bounded in $M$ is  $\delta$-approximated in $M$. Fix such an $x$, and let $a\in M$ be of size less than $\delta$. We have
		$|\mathcal P(a)|\leq 2^{<\delta}$, and hence
		$x\cap a\in \mathcal P(a)\subseteq M$.
	\end{proof}
	
	This  above proposition was also observed by Hachtman and Sinapova \cite{Hachtman-Sinapova2020}.

	\begin{lemma}[Krueger \cite{Krueger2020}]\label{K-nice-result}
		Assume that  $M$ is a $\delta^+$-guessing model with $\delta^+\subseteq M$. Suppose that $\mu\leq\delta$ is a regular cardinal and $M^{<\mu}\subseteq M$. Then, for every  $x\subseteq [M]^\mu$ which is bounded in $M$, there is  $y\in M\cap [M]^{\leq\delta}$ with
		$x\subseteq y$.
	\end{lemma}
	\begin{proof}
		Pick $X\in M$ with $x\subseteq X$.
		Fix a bijection $f: \mu\rightarrow x$, and let
		$x_\zeta=f[\zeta]$.
		By the closure property of $M$, we have $\mathfrak{X}\coloneqq\{x_\zeta:\zeta<\mu\}\subseteq M$. Note that 
		$\mathfrak{X}$ is bounded in $M$, as witnessed by
		$\mathfrak X\subseteq [X]^{<\mu}$.
		If $\mathfrak{X}$ is $\delta^+$-approximated in $M$, then it is guessed in $M$, and hence there is $\mathfrak{X}^*\in M$, with $\mathfrak{X}^*\cap M=\mathfrak{X}\cap M=\mathfrak X$. But then
		$x=\bigcup \mathfrak X^*\in M$. To see this, observe that if $\mathfrak X^*\nsubseteq M$, then $\delta^+\leq|\mathfrak{X}^*\cap M|=|\mathfrak{X}|=\mu$, which is a contradiction. Thus $\mathfrak{X}^*\subseteq M$, and hence $\mathfrak{X}^*=\mathfrak X$, which in turn implies that $x=\bigcup\mathfrak{X}^*\in M$.
		On the other hand, if $\mathfrak X$ is not $\delta^+$-approximated in $M$, then there is $Y\in M$ of size $\leq\delta$ such that $Y\cap \mathfrak{X}\notin M$. We may assume that
		$Y\subseteq [X]^{<\mu}$. By the closure property of $M$, we have $|Y\cap \mathfrak{X}|=\mu$. As $\mu$ is regular and $\mathfrak X$ is increasing, we have  $x\subseteq y= \bigcup Y$. It is easily seen that $|y|\leq\delta$.
	\end{proof}

	\begin{corollary}\label{Kruegerlemma}
		Assume that  $M$ is a $\delta^+$-guessing model with $\delta^+\subseteq M$. Suppose that $\delta$ is a regular cardinal and $M^{<\delta}\subseteq M$.  Then $(M,V)$ has the $\delta$-covering property. In particular,
		\begin{enumerate}
			\item if $M$ is an $\omega_1$-guessing of size $\omega_1$, then $M$ is internally unbounded.
			\item if $M$ is $\omega_1$-guessing in $M_1$ and $M_1$ is $\omega_1$-guessing in $M_2$, then
			$M$ is an $\omega_1$-guessing in $M_2$.
		\end{enumerate}
	\end{corollary}
	\begin{proof}[\nopunct]
	\end{proof}
	The last item above implies that having the   $\omega_1$-approximation for pairs of transitive models is a transitive relation.

	\subsubsection*{Viale's isomorphism theorem}
	It is not hard to prove the following interesting theorem for $0$-\gesms. It shows how much \gesms~behave like elementary submodels in Magidor's characterization of supercompactness \cite{MA1971}. Thus, in the following theorem, $\kappa$ has a compactness  property similar to that of a supercompact cardinal.

	\begin{theorem}[Viale \cite{V2012}]
		Assume  $M_i\prec H_{\theta_i}$, for $i=0,1$, are $\delta$-guessing. Suppose that
		\begin{enumerate}
			\item $\kappa:=\kappa_{M_0}=\kappa_{M_1}$ and $M_0\cap\kappa=M_1\cap\kappa$,
			\item $\mathcal P(\delta)\cap M_0=\mathcal P(\delta)\cap M_1$, and
			\item ${\rm o.t}({\rm Card}_{M_0})={\rm o.t}({\rm Card}_{M_1})$.
		\end{enumerate}
		Then, $(M_0,\in)\cong(M_1,\in)$.
	\end{theorem}
	\begin{proof}[\nopunct]
	\end{proof}
	
	It follows from Viale's proof (see \cite[Remark 5.4]{V2012} ) that, under  $T\coloneqq\rm PFA+(\ast)$, the third clause can be replaced by $\omega_1\subseteq M_0\cap M_1$. The consistency of  $T$ follows form the recent breakthrough  \cite{Asp-Sch} by Aspero and Schindler.

	\subsection{Open problems}
	In the definition of a $\delta$-\gesm, we  assumed that $\delta\leq\kappa_M$ is regular. 
	On the other hand, if we let $\delta\leq\kappa_M$  be a strong limit singular cardinal of countable cofinality, then \cref{0-g} implies that a $\delta$-\gesm~with $\delta\subseteq M$ must be a $0$-\gesm~and $2^{\delta}\leq|M|$. 
	
	Two imprecise questions come to my mind:
	
	\begin{question}
		Is it possible to introduce a useful notion of a $\delta$-\gesm~for $\delta>\kappa_M$ in a nontrivial way? What will be the relation between $\delta$ and the size of the model?
	\end{question}

	\begin{question}
		Is it possible to include singular cardinals (in particular non-strong limit cardinals) in the definition of a \gesm~and prove interesting results about them?
	\end{question}

	\section{GMP}

	As we shall see the existence of a single ``nice'' \gesm~has local impacts on the universe. The existence of a stationary set of \gesms~has even  global impacts on the entire universe. 
	We are now ready to introduce the simplest principle in terms of \gesms.
	
	For a powerful model $\mathcal H$, a regular cardinal $\delta\in\mathcal H$ and a cardinal $\kappa$, we let 
	
	\[
	\mathfrak G_{\kappa,\delta}(\mathcal H) = \{ M\in {\mathcal P}_\kappa(\mathcal H): M\prec \mathcal H  \mbox{ is a $\delta$-\gesm} \}. 
	\]

	\begin{definition}\label{GMP}
		${\rm GMP}(\kappa,\delta,\mathcal H)$ is the statement that  
		${\mathfrak G}_{\kappa,\delta}(\mathcal H)$ is stationary in ${\mathcal P}_{\kappa}(\mathcal H)$, and
		${\rm GMP}(\kappa,\delta)$ is the statement that ${\rm GMP}(\kappa,\delta, H_\theta)$ holds,
		for all sufficiently large regular $\theta$.\footnote{$\rm GMP$ stands for Guessing Model Principle.
		}
	\end{definition}

	Therefore  ${\rm GMP}(\kappa, \delta)$ implies ${\rm GMP}(\kappa, \delta')$ whenever  $\delta\leq \delta'$.
	Note that  if we have a  powerful filtration of the universe, i.e., a $\subseteq$-increasing sequence $\langle\mathcal H(\alpha):\alpha\in {\rm ORD}\rangle$ of powerful  set models with $V=\bigcup\{\mathcal H(\alpha):\alpha\in \ORD\}$, then ${\rm GMP}(\kappa,\delta)$ holds if and only if for all sufficiently large ordinals $\alpha$, ${\rm GMP}(\kappa,\delta,\mathcal H(\alpha))$ holds.
	
	\begin{convention}
		
		For a property $\Phi$, let us say that $\rm GMP(\kappa,\delta)$ is witnessed by $\Phi$-models or by models with  $\Phi$  if the models in $\mathfrak G_{\kappa,\gamma}(\mathcal H)$ with the property $\Phi$ forms a stationary set in
		$\mathcal P_\kappa(\mathcal H)$.
	\end{convention}

	\subsection{Consequences}
	
	\subsubsection*{Trees}
	For  a tree $T=(T,<_T)$ and a node $t\in T$, we let $b_t\coloneqq\{s\in T:s<_T t\}$. 
	Let $\kappa$ be an infinite cardinal. Recall that a tree $T$ of height  $\kappa$ is called a $\kappa$-tree if for every $\alpha<{\rm ht}(T)$, $|T_\alpha|<\kappa$.   A $\kappa$-Aronszajn tree is a $\kappa$-tree  without cofinal branches.  A cardinal $\kappa$ has the tree property (${\rm TP}(\kappa)$) if there is no $\kappa$-Aronszajn tree. A $\kappa$-Kurepa tree is a $\kappa$-tree with more than $\kappa$ cofinal branches. A weak $\kappa$-Kurepa tree is a tree of height and size $\kappa$ with more than $\kappa$ cofinal branches. The weak Kurepa Hypothesis at $\kappa$ ($w{\rm KH}(\kappa)$) states that there is a weak $\kappa$-Kurepa tree.

	\begin{lemma}\label{wKHlemma}
		Suppose $M$ is a $\delta$-\gesm. Assume that
		$T\in M$ is a tree of height $\delta$ with  $|T|< \kappa_M$. Then every cofinal branch through $T$ is in $M$. 
	\end{lemma}
	\begin{proof}
		Observe that  $T\subseteq M$, as $|T|<\kappa_M$.A   cofinal branch $b$  through $T$  is bounded in $M$ and satisfies $b\subseteq M$.
		If $a\in M\cap \mathcal P_\delta(M)$, then for some ordinal $\delta'<\delta$, $a\cap T\subseteq T_{<\delta'}$, i.e., every node in $a\cap T$ has height below $\delta'$.
		Pick $t\in b$ of height $\delta'$. Thus $t\in M$ and 
		\[
		b\cap a=\{s\in T: s<_Tt\}\cap a\in M.
		\]
		
		Thus if $b\subseteq T$ is a cofinal branch, then $b$ is $\delta$-approximated in $M$, and hence there is $b^*\in M$ such that 
		$b^*\cap M=b\cap M=b$. 
		Note that $b^*\subseteq M$. Therefore, $b=b^*\in M$.
	\end{proof}
	
	The above proof  is essentially a proof of the following theorem.
	
	\begin{theorem}[Cox--Krueger \cite{CK2017}]
		Assume  $\rm GMP(\delta^+,\delta)$. Then the weak Kurepa Hypothesis fails at $\delta$.
	\end{theorem}
	\begin{proof}[\nopunct]
	\end{proof}
	
	One can prove a slightly stronger result that is, under  $\rm GMP(\kappa,\delta)$ with $\kappa>\delta$,  every tree of height $\delta$ and size ${<}\kappa$ has less than $\kappa$ cofinal branches.
	Observe that  ${\rm GMP}(\omega_2,\omega_1)$ implies the failure of the weak Kurepa Hypothesis that was known to be a consequence of $\rm PFA$, and that   implies the failure of $\rm CH$.
	
	\begin{lemma}\label{TP}
		Suppose $M$ is a $\kappa_M$-\gesm. Then there is no $\kappa_M$-Aronszajn tree in $M$.
	\end{lemma}
	\begin{proof}
		Let $\kappa=\kappa_M$ and $\gamma={\rm sup}(M\cap\kappa_M)$.
		Note that $\kappa$ is a regular cardinal.
		Suppose that $T\in M$ is a $\kappa$-tree, and fix $t\in T$ of height $\gamma$. Consider  $b_t=\{s\in T: s<_T t\}$, which is bounded in $M$. We claim that $b_t$ is $\kappa$-approximated in $M$.
		Suppose $a\in M$ is of size less than $\kappa$, then, as in \cref{wKHlemma}, there   $\eta<\gamma$ with $a\subseteq T_{<\eta}$.
		Now pick $t\in  T_{\eta}\cap b$.  Observe that $t\in M$, since $T$ is a $\kappa_M$-tree in $M$. We have 
		\[
		b\cap a=a \cap \{s\in T:s<_T t\}\in M.
		\]
		Thus there is $b^*\in M$ with $b^*\cap M=b\cap M$. By elementarity, one can easily show that $b^*$ is a cofinal branch through $T$. Therefore, $T$ is not a $\kappa_M$-Aronszajn tree.
	\end{proof}
	
	The following is immediate.
	\begin{theorem}
		Assume  $\rm GMP(\delta,\delta)$. Then $\delta$ has the tree property.
	\end{theorem}
	\begin{proof}[\nopunct]
	\end{proof}
	
	Thus ${\rm GMP}(\delta^+,\delta)$ implies the tree property at $\delta^+$. In particular, ${\rm GMP}(\omega_2,\omega_1)$ implies the tree property at $\omega_2$, a consequence of $\rm PFA$.
	
	Recall that a  a weak ascent path of width a nonzero ordinal $\gamma $ through a tree $(T,<_T)$ is a sequence  $\langle\vec{t}^\alpha:\alpha<{\rm ht}(T)\rangle$ such that:
	\begin{itemize}
		\item For every $\alpha<{\rm ht}(T)$, $\vec{t}^\alpha:\gamma\rightarrow T_\alpha$ is a function, and
		\item for every $\alpha<\beta<{\rm ht}(T)$, there are $\zeta,\eta<\gamma$ such that $\vec{t}^\alpha(\zeta)<_T\vec{t}^\beta(\eta)$.
	\end{itemize}

	\begin{theorem}[Lambie-Hanson \cite{Lambie-Hanson-Square} / Lücke \cite{Luck-Ascend}]
		Assume ${\rm GMP}(\delta^+,\delta^+)$ is witnessed by  models which are $\delta$-close to $\delta^+$. Suppose $T$ is a tree of height $\delta^+$. Then $T$ has no weak ascent path of length ${<}\delta$ if and only if $T$ has no cofinal branches.
	\end{theorem}
	\begin{proof}
		The ``only if'' part is trivial. Let us assume that $f=\langle\vec{t}^\alpha:\alpha<\delta^+\rangle$ is an ascent path of width $\gamma<\delta$. We shall produce a cofinal branch through $T$.
		Pick a ${<}\delta$-closed model with $T,f\in M\prec H_\theta$. Observe that by \cref{K-nice-result}, $\alpha=M\cap\delta^+$ is of cofinality $\delta$, and that $\gamma\in M$.
		Thus, by a standard counting argument, there is $\zeta<\gamma$ such that $b=\{s\in T: s<_T t^\alpha(\zeta)\}\subseteq M$. As in \cref{TP}, $b$ is $\delta^+$-approximated in $M$, and hence guessed. Let $b^*\cap M=b$. By elementarity, $b^*$ is a cofinal branch through $T$.
	\end{proof}
	Recall that a set of size $\omega_1$ is  internally club  if it is the union of a $\subseteq$-continuous $\in$-sequence of countable sets of length $\omega_1$.
	Under $\rm PFA$, ${\rm GMP}(\omega_2,\omega_1)$ is witnessed by  internally club models (I.C.) models, and hence they have correct cofinality. Thus, under $\rm PFA$, every tree of height $\omega_2$ has a cofinal branch if and only if it has a weak ascent path of width $\omega$.
	The author \cite{sp-finite} has shown that if 
	${\rm GMP}(\omega_2,\omega_1)$ is witnessed by  I.C. models (and hence under $\rm PFA$), and  $T$  is a branchless tree of height $\omega_2$, then there is a proper and $\omega_2$-preserving property forcing notion $\mathbb P_T$ with the $\omega_1$-approximation property  that  specializes $T$.

	\subsubsection*{Cardinal arithmetic}
	The binary tree $T=2^{<\delta}$ has $2^{\delta}$  cofinal branches. Thus by applying   \cref{wKHlemma} to $T$,
	we obtain the following corollary.
	\begin{theorem}\label{card1}
		If there is a $\delta$-\gesm~$M$ of size $\kappa<2^{\delta}$ with $\kappa\subseteq M$. Then $\kappa< 2^{<\delta}$. In particular, if there is a $\gamma^+$-\gesm~$M$ of size $\gamma^+$ with $\gamma^+\subseteq M$ then $2^\gamma>\gamma^+$.
	\end{theorem}
	\begin{proof}[\nopunct]
	\end{proof}
	
	In fact, the existence of \gesms~is in favour of pushing up the relevant values of the continuum function. It is a theorem of Cox and Krueger \cite{CK-Quotient} that ${\rm GMP}(\omega_2,\omega_1)$ is consistent with the continuum being arbitrarily large.
	
	Recall that the Singular Cardinal Hypothesis ($\rm SCH$) states that if $\kappa$ is a singular cardinal with $2^{{\rm cof}(\kappa)}<\kappa$, then $\kappa^{{\rm cof}(\kappa)}=\kappa^+$.
	In \cite{V2012}, Vaile proved if ${\rm GMP}(\omega_2,\omega_1)$ is witnessed by internally unbounded models, then $\rm SCH$ holds. Later, Krueger proved (see \cref{Kruegerlemma} ) that every $\omega_1$-\gesm~of size $\omega_1$ is internally unbounded. 
	
	\begin{theorem}[Krueger \cite{Krueger2020}+Viale \cite{V2012}]
		${\rm GMP}(\omega_2,\omega_1)$ implies $\rm SCH$.
	\end{theorem}
	\begin{proof}[\nopunct]
	\end{proof}
	
	Thus, another consequence of $\rm PFA$ follows from
	${\rm GMP}(\omega_2,\omega_1)$.
	More generally, Krueger proved the following.

	\begin{theorem}[Krueger \cite{Krueger2020}] Assume ${\rm GMP}(\kappa,\omega_1)$. The Singular Cardinal Hypothesis holds above $\kappa$.
	\end{theorem}
	\begin{proof}[\nopunct]
	\end{proof}

	\subsubsection*{Squares}
	
	Let us recall the two-cardinal square principle.
	\begin{definition}
		Assume that $\kappa>0$ is a cardinal and $\lambda$ is an uncountable cardinal.
		Suppose that $S\subseteq\lambda$ is stationary.
		A sequence $\langle\mathcal C_\alpha: \alpha\in S \rangle$ is called a $\square(\kappa,\lambda,S)$-sequence if the following hold.
		\begin{enumerate}
			\item $\forall \alpha\in S$, $0<|\mathcal C_\alpha|<\kappa$,
			\item $\forall \alpha\in S$ and $\forall C\in\mathcal C_\alpha$, $C$ is a club in $\alpha$,
			\item $\forall \alpha\in S$, $\forall C\in\mathcal C_\alpha$, and $\forall \beta\in {\rm Lim}(C)$, $C\cap\beta\in \mathcal C_\beta$, and
			\item there is no club $D\subseteq\lambda$ so that for every $\alpha\in {\rm Lim}(D)\cap S$, $D\cap\alpha\in \mathcal C_\alpha$.
		\end{enumerate}
	\end{definition}
	
	Let $S^\lambda_\kappa\coloneqq\{\alpha\in\lambda:{\rm cof}(\alpha)=\kappa\}$. The set $S^\lambda_{<\kappa}$ is defined naturally.
	
	\begin{theorem}[Wei\ss ~\cite{W2010, W2012}, Vaile \cite{V2012}] 
		Assume ${\rm GMP}(\kappa,\kappa)$. Then  $\square(\kappa,\lambda,S^\lambda_{<\kappa})$ fails, for every $\lambda$ with ${\rm cof}(\lambda)\geq\kappa$.
	\end{theorem}
	\begin{proof}
		Assume that there exists a $\square(\kappa,\lambda,S^\lambda_{<\kappa})$-sequence $\mathscr C=\langle\mathcal C_\alpha:\alpha\in S^\lambda_{<\kappa}\rangle$. We shall find a contradiction.  Pick a $\kappa$-\gesm~$M\prec H_\theta$, for some large regular $\theta$, with $\mathscr C,\kappa,\lambda\in M$ and $M\cap\kappa\in\kappa$. Let $\delta={\rm sup}(M\cap\lambda)$.
		Thus $\delta\in S^\lambda_{<\kappa}$.
		
		Fix $C\in\mathcal C_\delta$.
		We claim that $C$ is $\kappa$-approximated in $M$.
		Assume that $a\in M$ is of size less than $\kappa$. Note that $a\subseteq M$.
		We may assume that $a\subseteq\lambda$. Since the cofinality of $\lambda$ is  at least $\kappa$, we have ${\rm sup}(a)<\lambda$.
		Let $\langle\alpha_\xi:\xi<\kappa'<\kappa\rangle$ be the increasing enumeration of $a$. We may assume that $C\cap a$ is infinite.
		Let $\alpha$ be the largest element of $C$ that is a limit point of  $C\cap\alpha$. Observe that  $C\cap a\setminus\alpha$ is finite, and that $\alpha<\delta$. Thus it is enough to show that $C\cap a\cap\alpha$ is in $M$.
		There is $\xi<\kappa'$ such that $\alpha={\rm sup}\{\alpha_\eta:\eta<\xi\}$. Thus $\alpha\in M$.
		Now, $\alpha$ is a limit point of $C$, and hence $C\cap\alpha\in \mathcal C_\alpha$. Since $M\cap\kappa$ is an ordinal, $\alpha\in M$,  and $|\mathcal C_\alpha|<\kappa$, we have $\mathcal C_\alpha\subseteq M$, and thus $C\cap\alpha\in M$.
		
		Since $M$ is $\kappa$-guessing, there is $C^*\in M$ such that
		$C^*\cap M=C\cap M$. By elementarity, $C^*$  is a club relative to $S^\lambda_{<\kappa}$. Observe that if $\gamma$ is a limit point of $C^*$ of cofinality ${<}\kappa$, then 
		\[
		C^*\cap\gamma\cap M=C\cap\gamma\cap M.
		\]
		On the other hand, there is $F\in\mathcal C_\gamma$ such that $F=C\cap\gamma$.
		Now $F\in M$ and $C^*\cap M\cap\gamma=F\cap M$. By elementarity, $C^*\cap\alpha=F$, and hence $C^*\cap\gamma\in\mathcal C_\gamma$. 
		Let $D$ be the closure of $C^*$. Then $D\in M$.
		Now if $\gamma\in M$ if of cofinality ${<}\kappa$, and is a limit point of $D$, we then have $\gamma\in C^*$ and $D\cap\gamma=C^*\cap\gamma\in\mathcal C_\gamma$. This is a contradiction as $\mathscr C$ is a $\square(\kappa,\lambda,S^\lambda_{<{\rm cof}(\kappa)})$-sequence.
		
	\end{proof}

	A similar proof shows that following.
	\begin{theorem}[Wei\ss~\cite{W2010, W2012}] 
		Assume ${\rm GMP}(\kappa,\kappa^+)$ is witnessed by models which have  bounded uniform cofinality $\delta$. Then  $\square(\kappa,\lambda,S^\lambda_\delta)$ fails, for every $\lambda$ with ${\rm cof}(\lambda)\geq\kappa$.
	\end{theorem}
	\begin{proof}[\nopunct]
	\end{proof}
	
	\begin{corollary}
		Assume ${\rm GMP}(\kappa,\delta)$ is witnessed by models ${<}\delta$-closed models. Then  $\square_{E_{\delta}}(\kappa,\lambda)$ fails, for every $\lambda$ with ${\rm cof}(\lambda)\geq\kappa$.
	\end{corollary}
	\begin{proof}[\nopunct]
	\end{proof}
	
	\subsubsection*{The approachability ideal}
	\begin{definition}\label{appseq} Let $\lambda$ be a regular cardinal. A $\lambda$-{\em approaching sequence} is
		a $\lambda$-sequence of bounded subsets of $\lambda$. If $\bar a = (a_\xi: \xi < \lambda)$ is a $\lambda$-approaching sequence, 
		we let $B(\bar a)$ denote the set of all $\delta < \lambda$ such that 
		there is a cofinal subset $c\subseteq \delta$ such that:
		\begin{enumerate}
			\item ${\rm otp}(c) < \delta$, in particular $\delta$ is singular,  and
			\item for all $\gamma < \delta$, there exists $\eta < \delta$ such that $c \cap \gamma = a_\eta$.
		\end{enumerate}
	\end{definition}

	\begin{definition}\label{AppId}
		Suppose $\lambda$ is a regular cardinal. Let $I[\lambda]$ be the ideal generated by the sets $B(\bar a)$, 
		for all $\lambda$-approaching sequences $\bar a$,
		and the non stationary ideal ${\rm NS}_\lambda$.  
	\end{definition}
	
	\begin{definition}
		We say that the approachability holds at $\delta$ if $\delta^+\in I[\delta^+]$. We denote this principle by ${\rm AP}(\delta)$.
	\end{definition}

	\begin{definition}
		For a regular infinite cardinal $\kappa$, we let ${\rm MP}(\kappa^+)$ denote the following statement.  
		\[
		I[\kappa^+]\restriction S^{\kappa^+}_{\kappa}={\rm NS}_{\kappa^+}\rest S^{\kappa^+}_{\kappa}.
		\]
	\end{definition}
	
	\begin{theorem}[Wei\ss~\cite{W2010}]
		$\rm GMP(\delta^{+},\delta)$ implies $\neg {\rm AP}(\delta)$.
	\end{theorem}
	\begin{proof}
		See \cref{MP from FS }.
	\end{proof}

	\subsubsection*{Laver diamonds}

	The following beautiful theorem deserves more attention.

	\begin{theorem}[Viale \cite{V2012}]
		Assume  $\rm PFA$. Suppose there is a proper class of Woodin cardinals.
		There is a function $f:\omega_2\rightarrow H_{\omega_{2}}$ such that for every $\theta\geq\omega_2$ and every $x\in H_\theta$,
		\[
		\{M\in\mathfrak G_{\omega_2,\omega_1}(H_\theta): x\in M \mbox{ and } \pi_M(x)=f(M\cap\omega_2)\}
		\]
		is stationary in $\mathcal P_{\omega_2}(H_\theta)$.
	\end{theorem}

	\subsection{Consistency results}

	\begin{theorem}[Vaile--Wei{\ss} \cite{VW2011}]
		$\rm PFA$ implies that ${\rm GMP}(\omega_2,\omega_1)$ is witnessed by internally club models.
	\end{theorem}
	\begin{proof}
		See \cite[Theorem 4.4]{V2012}.
	\end{proof}
	Let us also mention that in \cite{T2016} Trang showed the consistency of ${\rm GMP}(\omega_3,\omega_2)$
	assuming the existence of a supercompact cardinal.  In his model the Continuum Hypothesis holds.
	Thus $\rm GMP(\omega_2,\omega_1)$ fails by \cref{card1}.
	
	The following is  well-known.
	
	\begin{theorem}
		Assume the $\kappa$ is a supercompact cardinal. Suppose that $\delta<\kappa$ is a regular cardinal. Then in a generic extension ${\rm GMP}(\delta^{++},\delta^{+})$ holds.
	\end{theorem}
	\begin{proof}[Proof (sketch)]
		We may assume without loss of generality that $\delta^{<\delta}=\delta$.
		We define two forcings either of which works equally.  Let  $\mathbb P_0$ be  Neeman's forcing  with sequences of models of length less than $\delta$, where transitive models are $V_\alpha\prec V_\kappa$ with ${\rm cof}(\alpha)\geq\delta$ and the nontransitive models are  ${<}\delta$-closed elementary substructures of $V_\kappa$ of size $\delta$ containing $\delta$ as an element, and let $\mathbb P_1$ be Veličković's forcing whose conditions are ${<}\delta$-sized sets of   ${<}\delta$-closed virtual models of size $\delta$ in the structure $(V_\kappa,\in,\delta)$.
		
		Let $\mathbb P$ be either of these two forcings.
		Then, $\mathbb P$ is a ${<}\delta$-closed and $\kappa$-c.c.  forcing which is strongly proper for models under consideration. Since $\kappa$ is supercompact, by Magidor's characterization \cite{MA1971}, there are stationary many  $0$-gesms  $M\prec V_\lambda$ of size less than $\kappa$, for every limit ordinal $\lambda$. Pick such a model in  $V_\lambda$ and assume that ${\rm cof}(\lambda)\geq\delta$.
		One has to show that if $G$ is a $V$-generic filter, then $M[G]$ is a $\delta$-\gesm~of size $\delta^+$.
		Of course, $\mathbb P$ forces $\kappa$ to be $\delta^{++}$, and hence $M[G]$ is forced to be of size $\delta^{+}$. 
		
		Let
		$\overline{M}=V_{\gamma}$.
		It is proved that $1_{\mathbb P}$ is strongly $(M,\mathbb P)$-generic and hence we have 
		\[
		\overline{M[G]} =V_\gamma^{V[G\cap M]}.
		\]
		Then, one needs to show that $\mathbb P/G\cap M$ is strongly proper for a stationary sets of ${<}\delta$-closed models of size $\delta$.
		and hence $(V[G\cap M],V[G])$ has the $\delta^+$-approximation property. 
		Applying \cref{CK-formulation,transitivity under covering}, we have $M[G]$ is a $\delta^+$-\gesm~in $V[G]$. 
	\end{proof}

	There is a surprising result around singular cardinals.
	
	\begin{theorem}[Hachtman--Sinapova \cite{Hachtman-Sinapova2019}]
		If $\delta$ is a countable limit of supercompact cardinals, then
		${\rm GMP}(\delta^+,\delta^+)$ holds.
	\end{theorem}
	\begin{proof}[\nopunct]
	\end{proof}
	\subsection{Gesms and ISP}\label{secISP}
	
	The statement ${\rm GMP}(\kappa,\omega_1)$ is a reformulation of the principle ${\rm ISP}(\kappa)$ introduced
	by  C. Wei{\ss} in \cite{W2010}.
	The equivalence between ${\rm GMP}(\kappa,\omega_1)$ and  ${\rm ISP}(\kappa)$ was established
	in \cite{VW2011}.
	
	Let $\delta\leq \kappa\leq\lambda$ be infinite cardinals with $\delta$ regular.
	Recall that  a $\mathcal P_\kappa(\lambda)$-list is a sequence
	$\langle d_a:a\in\mathcal P_\kappa(\lambda)\rangle$ such that $d_a\subseteq a$, for every  $a\in\mathcal P_\kappa(\lambda)$.
	
	\begin{definition}
		A list $\langle d_a:a\in\mathcal P_\kappa(\lambda)\rangle$  is called $\delta$-slender if for every sufficiently large  $\theta$, there is a club $C\subseteq\mathcal P_\kappa(H_\theta)$ such that for every $M\in C$ and every $a\in M$ with $a\in M\cap\mathcal P_\kappa(\lambda)$, $a\cap d_{M\cap\lambda}\in M$.
	\end{definition}
	
	\begin{definition}
		A set $d\subseteq \lambda$ is  an ineffable branch through a $\mathcal P_\kappa(\lambda)$-list $\langle d_a:a\in\mathcal P_\kappa(\lambda)\rangle$ if there is a stationary set $S\subseteq\mathcal P_\kappa(\lambda)$ such that for every $a\in S$, $a\cap d=d_a$
	\end{definition}
	
	\begin{definition}
		The principle ${\rm ISP}_\delta(\kappa,\lambda)$ states that every $\delta$-slender $\mathcal P_\kappa(\lambda)$-list has an ineffable branch.
		Let also ${\rm ISP}_\delta(\kappa)$ state that ${\rm ISP}_\delta(\kappa,\lambda)$ holds, for every $\lambda\geq\kappa$.
	\end{definition}
	
	Note that ${\rm ISP}_{\omega_1}(\kappa,\lambda)$ is the same as the well-known ${\rm ISP}(\kappa,\lambda)$. 
	
	\begin{proposition}[Viale--Weiss \cite{VW2011}]
		${\rm ISP}_\delta(\kappa)$ holds if and only if ${\rm GMP}(\kappa,\delta)$ holds.
	\end{proposition}
	\begin{proof}
		
		See \cite[Propositions 3.2 and 3.3]{VW2011}.

	\end{proof}
	
	The following is due to Magidor \cite{MA1974}.
	
	\begin{theorem}
		Let $\kappa$ be an inaccessible cardinal. Assume that ${\rm ISP_{\delta}}(\kappa)$ holds, for some regular $\delta<\kappa$. Then $\kappa$ is supercompact.
	\end{theorem}
	\begin{proof}[\nopunct]
	\end{proof}
	
	\subsection{Open problems}

	\begin{problem}[Viale \cite{V2012}]
		Is it consistent to have $\omega_2$-\gesms~of size $\omega_1$, which are not $\omega_1$-\gesm?
	\end{problem}
	We ask the following more general question.
	
	\begin{problem}
		Is it consistent to have $\delta$-\gesms~of size less than $\delta$ which are not $\gamma$-\gesm, for all $\gamma\leq|M|$?
	\end{problem}
	
	\begin{problem}
		Given $m,n\in\omega$ with $m\geq 3$ and $n\geq 1$. Is   ${\rm GMP}(\omega_{m},\omega_n)$ consistent? Of particular interest is 
		${\rm GMP}(\omega_{m},\omega_1)$.
	\end{problem}
	
	\begin{problem}
		It is consistent to have ${\rm GMP}(\omega_{n+1},\omega_n)$, for every $n\geq 1$?
	\end{problem}
	
	Of course, one can also ask the above questions about ${\rm GMP}(\delta^{+m},\delta^{+n})$.
	
	\begin{problem}
		Assume that $\delta>\omega$ is a  weakly, but not strongly, inaccessible cardinal. Is
		${\rm GMP}(\delta,\delta)$ consistent? 
	\end{problem}

	\[
	\adjustbox{scale=1.06,center}
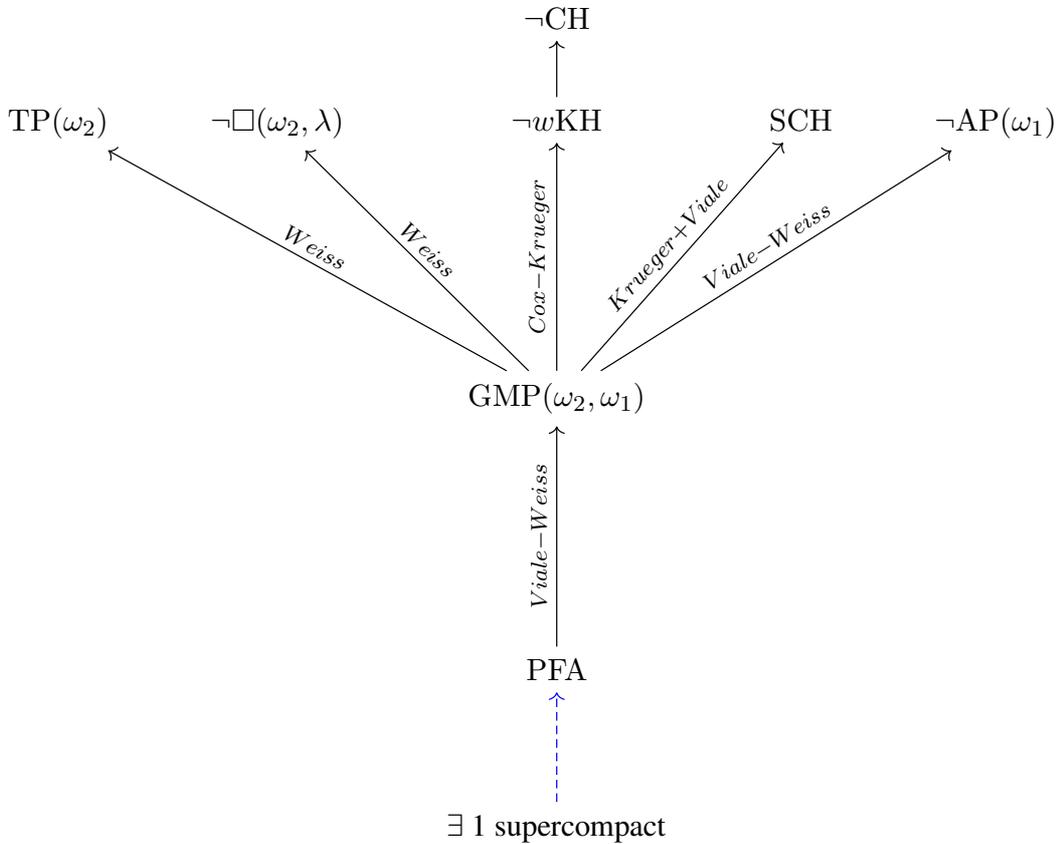
{
		\begin{tikzcd}
			& & \neg {\rm CH}&    & &  \\
			{\rm TP}(\omega_2) &   \neg \square(\omega_2,\lambda)& \neg w{\rm KH} \arrow[u] & \rm SCH &\neg {\rm AP}(\omega_1) &  \\
			& & & &  \\
			& & & &\\
			& &  & &\\
			&  &  {\rm GMP}(\omega_2,\omega_1)\arrow{uuuu}[sloped]{Cox-Krueger}\arrow{uuuul}[sloped]{ Weiss} \arrow{uuuur}[sloped]{ Krueger+Viale}\arrow{uuuurr}[sloped]{ Viale-Weiss}\arrow{uuuull} [sloped]{Weiss} &  & \\
			& & & &  \\
			& & & &  \\
			& & & &  \\
			& &  {\rm PFA} \arrow{uuuu}[sloped]{Viale-Weiss}   & & \\
			& & & &  \\
			& &  \exists~{\mbox{1 supercompact}} \arrow[sloped,blue,dashed]{uu}   & & 
		\end{tikzcd}
	}
	\]
	
	\clearpage
	
	\[
	\adjustbox{scale=1.1,center}
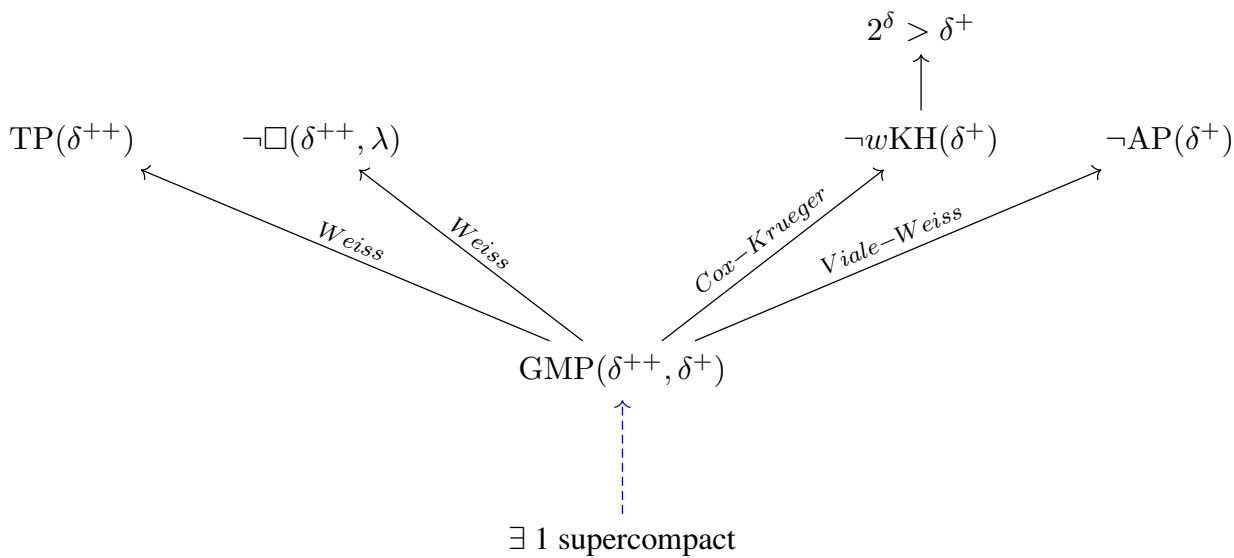
{
		\begin{tikzcd}
			& & & 2^{\delta}>\delta^+   & &  \\
			{\rm TP}(\delta^{++}) &   \neg \square(\delta^{++},\lambda)& &\neg w{\rm KH}(\delta^{+}) \arrow[u]  &\neg {\rm AP}(\delta^+) &  \\
			& & & &  \\
			& &  & &\\
			&  &  {\rm GMP}(\delta^{++},\delta^{+})\arrow{uuur}[sloped]{Cox-Krueger}\arrow{uuul}[sloped]{ Weiss}\arrow{uuurr}[sloped]{ Viale-Weiss}\arrow{uuull} [sloped]{Weiss} &  & \\
			& & & &  \\
			& &\exists~{\mbox{1 supercompact}} \arrow[blue,dashed]{uu}  & & 
		\end{tikzcd}
	}
	\]
	
	\section{\texorpdfstring{${\mbox{\textbf{GMP}}}^{\bm{+}}$}{}}
	As we have seen before, one of the consequence of $\rm GMP(\omega_2,\omega_1)$ was the failure of the approachability  property at $\omega_1$. However, $\rm GMP(\omega_2,\omega_1)$ does not imply $\rm MP(\omega_2)$ since
	$\rm GMP(\omega_2,\omega_1)$ is consistent with $2^{\aleph_0}=\aleph_2$, but $\rm MP(\omega_2)$ implies $2^{\aleph_0}\geq\aleph_3$.
	We shall introduce a certain  strengthening  of ${\rm GMP}(\omega_2,\omega_1)$ that implies 
	$\rm MP(\omega_2)$.
	
	\subsection{Strongly guessing models}
	
	We are now about to define a stronger version of guessing models by imposing constraints on their structure. 
	
	\begin{definition}[Mohammadpour--Veličković \cite{MV}]\label{Strong guessing}
		Let   $\delta \leq \kappa$  be regular uncountable cardinals. A model  $M$ of cardinality $\kappa^+$
		is called \textbf{strongly $\bm{\delta}$-\gesm} if it is the union of an $\in$-increasing chain $\langle M_\xi : \xi <\kappa^+\rangle$
		of $\delta$-\gesms~of cardinality $\kappa$ with $M_\xi= \bigcup \{ M_\eta : \eta < \xi\}$, for every $\xi$ of cofinality $\kappa$. 
	\end{definition}
	
	\begin{lemma}\label{strong guessing is guessing}
		Every strongly $\delta$-\gesm~is $\delta$-guessing. 
	\end{lemma}
	\begin{proof}
		Suppose that $M$ is a strongly $\delta$-\gesm~that is witnessed by a sequence $\langle M_\xi : \xi <\kappa^+\rangle $.
		Suppose that $A$ is bounded in $M$, we may assume that $A$ is bounded in each $M_\xi$.
		Since the sequence $(M_\xi : \xi <\kappa^+)$ is closed at ordinals of cofinality $\kappa$, 
		a standard closure argument shows that the following set, modulo $S^\kappa_{\kappa^+}$, is a club in $\kappa^+$.
		
		$$C=\{\xi<\kappa^+: A \text{ is } \delta\text{-approximated in } M_\xi\}$$
		
		Thus for each $\xi\in C$, there is $A_\xi\in M_\xi$ such that $A_\xi\cap M=A\cap M_\xi$.
		By Fodor's lemma, there is some $\eta<\kappa^+$ and a stationary set $S\subseteq C$ such that 
		for each $\xi\in S$, $A_\xi$ is in $M_\eta$. On the other hand $|M_\eta|<\kappa^+$, and hence 
		there is a stationary set $T\subseteq S$, such that for every $\xi,\xi^\prime \in T$, $A_\xi=A_{\xi^\prime}$.
		Let $A^*=A_\xi$, for some $\xi\in T$.
		It is easy to see that
		$A^*\cap M=A\cap M$.
		
	\end{proof}

	As before, for a powerful model $\mathcal H$ and regular cardinals $\delta\in \mathcal H$ and $\kappa$, we let 
	
	\[
	\mathfrak G^+_{\kappa^{++},\delta}(\mathcal H) = \{ M\in {\mathcal P}_{\kappa^{++}}(\mathcal H): M\prec \mathcal H  \mbox{ is strongly $\delta$-\gesm} \}. 
	\]
	
	\begin{definition}[Mohammadpour--Veličković \cite{MV}]
		Let ${\rm GMP}^+(\kappa^{++},\delta)$ states that for every sufficiently large regular cardinal $\theta$,  $\mathfrak G^+_{\kappa^{++},\delta}(H_\theta)$ is stationary in $\mathcal P_{\kappa^{++}}(H_\theta)$.
	\end{definition}

	We have the following immediate fact.
	\begin{fact}
		Assume  ${\rm GMP}^+(\kappa^{++},\delta)$. Then 
		both ${\rm GMP}(\kappa^+,\delta)$ and ${\rm GMP}(\kappa^{++},\delta)$ hold.
	\end{fact}

	\begin{theorem}[Mohammadpour--Veličković \cite{MV}]\label{MP from FS }
		${\rm GMP}^+(\kappa^{++},\kappa)$
		implies ${\rm MP}(\kappa^+)$. 
	\end{theorem}
	
	\begin{proof}
		Let $\bar a=\langle a_\xi : \xi <\kappa^+\rangle $ be a $\kappa^+$-approaching sequence which belongs
		to $H_{\kappa^{++}}$. 
		Let $\mathcal G$ be the (stationary) set of $\kappa$-\gesms~containing $\bar a$.
		We  show that  $M\cap \kappa^+$ is not in  $B(\bar a)$, for any $M\in \mathcal G$ such that $\cof(M\cap \kappa^+)=\kappa$. 
		Fix one such $M\in \mathcal G$. Let $\delta = M\cap \kappa^+$ and suppose that $c\subseteq \delta$ satisfies 
		(1) and (2) of \cref{appseq}. Let $\mu = {\rm otp}(c)$. Note that $\mu< \delta$, hence $\mu \in M$. 
		Since $\bar a\in M$, we have that $c\cap \gamma \in M$, for all $\gamma <\delta$, and hence $c\cap Z\in M$, for all $Z\in M$ with $|Z|<\kappa$.
		Since $M$ is a $\kappa$-\gesm,  there must be $d\in M$ such that $c = d\cap \delta$. We may assume that $d\subseteq \kappa^+$.
		Then $c$ is an initial segment of $d$, so if $\rho$ is the $\mu$-th element of $d$ then $d\cap \rho = c$.
		Since $\mu, d \in M$, we have $\rho\in M$ as well, and hence $c=d \cap \rho \in M$. 
		But then $\delta= \sup (c)$ belongs to $M$, a contradiction. 
	\end{proof}

	\begin{theorem}[Mohammadpour--Veličković \cite{MV}]
		Suppose $\kappa$ is a regular cardinal. 
		Assume there are two supercompact cardinals above $\kappa$. Then ${\rm GMP}^+(\kappa^{+++},\kappa^+)$ holds in a ${<}\kappa$-closed generic extension.
	\end{theorem}
	\begin{proof}[\nopunct]
	\end{proof}
	
	Assume $\lambda<\mu$ are supercompact cardinals above $\kappa$. The proof of the above theorem uses a forcing with ${<}\kappa$-sized decorated chains of virtual models of two types: countable and $\lambda$-Magidor models.
	The forcing is $\mu$-Knaster and strongly proper for the both collections of models.
	
	\begin{definition}
		Let   $\delta \leq \kappa$  be regular uncountable cardinals. For $n\in\mathbb N$, a model  $M$ of cardinality $\kappa^{+(n+1)}$
		is called an ${\bm n}$\textbf{-strongly $\bm{\delta}$-guessing model}  if it is the union of an $\in$-increasing sequence $\langle M_\xi : \xi <\kappa^+\rangle$
		of  $(n-1)$-strongly $\delta$-guessing models of cardinality $\kappa^{+n}$ such that  for every $\xi$ of cofinality $\kappa^{+n}$,  $M_\xi= \bigcup \{ M_\eta : \eta < \xi\}$, 
	\end{definition}
	In the above definition, if $n=1$, then an inductively strongly $\delta$-guessing model is just a strongly $\delta$-guessing model. 
	The principle  ${\rm GMP}^{+n}(\kappa^{+n},\delta)$ is defined in the obvious meaning.
	Let also $\omega$-strongly $\delta$-guessing models be defined in the obvious way with ${\rm GMP}^{\infty}(\kappa^{+\omega},\delta)$ for the associated principle.
	\subsection{Open problems}

	\begin{problem}
		Does ${\rm GMP}(\omega_3,\omega_1)$ imply ${\rm MP}(\omega_2)$?
	\end{problem} 
	\begin{problem}
		Given $n\geq 2$. Is ${\rm GMP}^{+n}(\kappa^{+n},\delta)$ consistent?
	\end{problem}
	
	\begin{problem}
		Is ${\rm GMP}^{\infty}(\aleph_{\omega},\omega_1)$ consistent?
	\end{problem}

	\[
	\adjustbox{scale=.97,center}
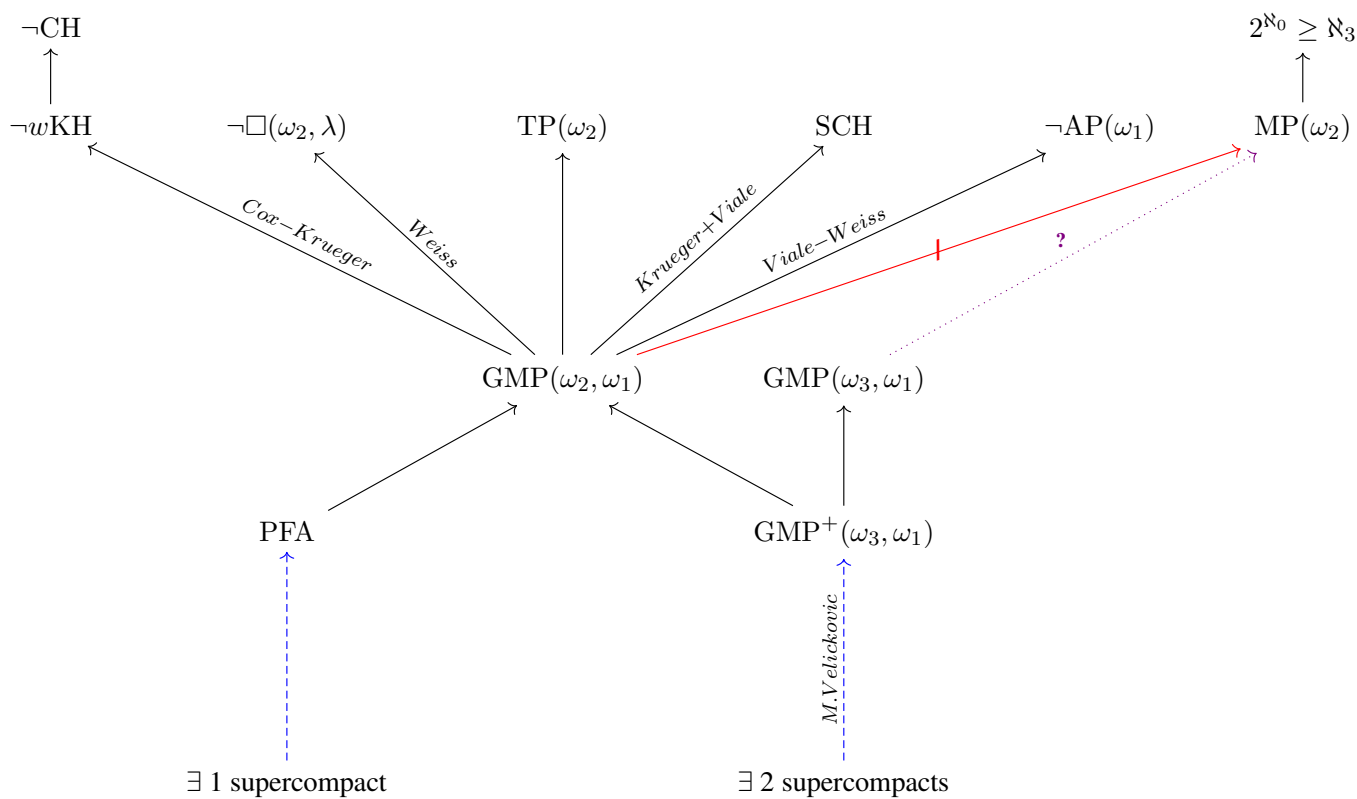
{%
		\begin{tikzcd}
			\neg {\rm CH}&  & &  & &2^{\aleph_0}\geq \aleph_3  \\
			\neg w{\rm KH} \arrow[u]&   \neg \square(\omega_2,\lambda)& \rm TP(\omega_2) & \rm SCH &\neg {\rm AP}(\omega_1) & {\rm MP}(\omega_2)\arrow[u] \\
			& & & &  \\
			& & & &\\
			& &  & &\\
			&  &  {\rm GMP}(\omega_2,\omega_1)\arrow[uuuu]\arrow{uuuul}[sloped]{ Weiss} \arrow{uuuur}[sloped]{ Krueger+Viale}\arrow{uuuurr}[sloped]{ Viale-Weiss}\arrow{uuuull}[sloped]{ Cox-Krueger}\arrow[uuuurrr, red, "\textcolor{red}{ \textbf{/}}" marking]  & {\rm GMP}(\omega_3,\omega_1)
			\arrow[uuuurr, violet, dotted, "\textcolor{violet}{ \textbf{?}}" ] & \\
			& & & &  \\
			&   {\rm PFA} \arrow{uur}  & & {\rm GMP}^+(\omega_3,\omega_1) \arrow{uul} \arrow{uu} & \\
			& & & &  \\
			& & & &  \\ 
			& & & &  \\
			&   \exists~\mbox{1 supercompact} \arrow[blue,dashed]{uuuu} & & 
			\exists~\mbox{2 supercompacts} \arrow[blue,dashed]{uuuu}[sloped,black]{M.Velickovic}  & 
		\end{tikzcd}
	}
	\]

	\section{IGMP}

	The indestructible version of guessing models was first discovered and studied by Cox and Krueger \cite{CK2017}.
	A $\omega_1$-guessing set is called \textbf{indestructible} if it remain  $\omega_1$-guessing  in any   $\omega_1$-preserving forcing extension. 
	More generally, we can require more robustness. For example, by requiring the size does not collapse, etc., or by considering  some other kind of extensions.
	
	\begin{definition}
		Suppose $M$ is a $\delta$-guessing model. Let $\mathcal C$ be a class of $\delta$-preserving forcing notions. Then $M$ is called $\mathcal C$-indestructible if for every $\mathbb P\in\mathcal C$, $M$ satisfies the  $\delta$-guessing property in generic extensions by $\mathbb P$.
	\end{definition}
	
	By $\mathbb P$-indestructible, we shall mean $\{\mathbb P\}$-indestructible. We omit $\mathcal C$ whenever it is the class of all $\delta$-preserving forcings.
	
	\begin{definition}
		Let $\mathcal C$ be a class of $\delta$-preserving forcings.
		Let $P$ be one of the above principles about $\delta$-guessing models. Then by $\mathcal C-{\rm IP}$, we mean that $P$ is witnessed by $\mathcal C$-indestructible models.
	\end{definition}
	Notice that the above definition is not  a priori equivalent to the statement that $\mathbb P$-$\rm IP$ holds, for every $\mathbb P\in\mathcal C$.

	\begin{definition}
		Let ${\rm IGMP}(\omega_2,\omega_1)$ state that for all sufficiently large regular cardinal $\theta$, the set of indestructible $\omega_1$-\gesms~of $H_\theta$ is stationary in $\mathcal P_{\omega_2}(H_\theta)$. 
	\end{definition}
	
	Let us start with the consistency of ${\rm IGMP}(\omega_2,\omega_1)$.
	
	\begin{theorem}[Cox--Krueger \cite{CK2017}]
		$\rm PFA$ implies ${\rm IGMP}(\omega_2,\omega_1)$.
	\end{theorem}
	\begin{proof}[\nopunct]
	\end{proof}
	In the same paper, Cox and Krueger also showed that ${\rm IGMP}(\omega_2,\omega_1)$ is  consistent with the continuum being arbitrarily large. 
	It is also worth mentioning that the ${\rm IGMP}(\omega_2,\omega_1)$ obtained by Cox and Kruger has the property that every indestructible $\omega_1$-guessing model remains $\omega_1$-guessing in any outer transitive extension with the same $\omega_1$.
	The main idea to make a guessing set  $M$ indestructible is to specialise a  branchless $\omega_1$-tall, $\omega_1$-sized tree $T(M)$ that is naturally associated to $M$ so that  $T(M)$   is special if and only if   the guessing property of $M$ is indestructible. Now if one successes to find a model in which every branchless tree of size and height $\omega_1$ is special and that there are $\omega_1$-guessing models of size $\omega_1$, then the guessing models are indestructible. A true conjunction  under $\rm PFA$.

	\subsection{Consequences}
	
	\subsubsection*{Indestructibility and the approximation property}
	One of the tight interaction of the above indestructibility is with the $\omega_1$-approximation property of forcings. Recall that a forcing notion has the 
	$\delta$-approximation property in $V$ if for every $V$-generic filter $G\subseteq\mathbb P$, $(V,V[G])$ has the $\delta$-approximation property.  Note that ${\rm GMP}(\omega_2,\omega_1)$ is consistent with the failure of the Suslin Hypothesis: In \cite{CK-Quotient}, Cox and Krueger obtained the consistency of ${\rm GMP}(\omega_2,\omega_1)$ with arbitrary large continuum in a way that the forcing adds Cohen reals over the ground model, and hence by a well-known result due to Shelah, there is a Suslin tree in the final model. There are also other models  witnessing this, see e.g. \cite{CK2017}.

	\begin{proposition}[Cox--Krueger \cite{CK2017}]\label{Special implies SH}
		$\rm IGMP(\omega_2,\omega_1)$ implies  the Suslin Hypothesis ($\rm SH$).
	\end{proposition}
	\begin{proof}[\nopunct]
	\end{proof}
	
	More generally, if one regards a tree of height and size $\omega_1$ with the opposite ordering as a forcing notion, then under $\rm IGMP(\omega_2,\omega_1)$, a nontrivial $T$ collapses $\omega_1$.
	\begin{lemma}\label{Lemma0}
		Suppose   $\mathbb P$ has the $\delta$-approximation property, and that $M\prec H_\theta$ is a $\delta$-gesm, for some $\theta\geq\omega_2$. 
		Assume that $(M,V)$ has the $\delta$-covering property.
		Then $\mathbb P$ forces  $M$ to  be $\delta$-guessing.
	\end{lemma}
	\begin{proof}
		Let $G\subseteq\mathbb P$ be a $V$-generic filter.
		Fix  $x\in V[G]$ and assume that $x\subseteq X\in M$ is
		$\delta$-approximated in $M$. We claim that $x\cap M$ is $\delta$-approximated in $V$, which in turn implies that $x\cap M\in V$. Then, since $M$ is a $\delta$-gesm in $V$, $x$ is guessed in $M$. 
		To see that $x\cap M$ is $\delta$-approximated in $V$, fix a set $a\in V$ of size less than $\delta$. By  the $\delta$-covering property,
		there is a  set $b\in M$ of size less than $\delta$ with $a\cap M\cap X\subseteq b$. Thus
		$a\cap x\cap M=a\cap x\cap b\in V$, since $a\in V$ and $x\cap b\in M\subseteq V$. 
	\end{proof}

	\begin{corollary}
		Suppose that  $\mathbb P$ has the $\delta^+$-approximation property. Assume that $M\prec H_\theta$ is $\delta^+$-guessing with $\delta^+\subseteq M$, for some $\theta\geq\omega_2$. 
		Then $\mathbb P$ forces  $M$ to  be $\delta^+$-guessing.
	\end{corollary}
	\begin{proof}
		We only need the covering property that holds by \cref{K-nice-result}.
	\end{proof}
	
	\begin{corollary}
		Suppose   $\mathbb P$ has the $\omega_1$-approximation property, and that $M\prec H_\theta$ is $\omega_1$-guessing for some $\theta\geq\omega_2$. 
		Then $\mathbb P$ forces  $M$ to  be $\omega_1$-guessing.
	\end{corollary}
	\begin{proof}[\nopunct]
	\end{proof}

	\begin{proposition}\label{prop0}
		Assume that $\mathbb P$ is an $\omega_1$-preserving forcing. Suppose that for every sufficiently large regular cardinal $\theta$, $\mathbb P$ is proper for a stationary set $\mathfrak{G}_\theta\subseteq\mathcal P_{\omega_2}(H_\theta)$ of
		$\delta$-\gesms~of $H_\theta$ which are $\mathbb P$-indestructible.\footnote{i.e., for every $M\in\mathfrak G_\theta$ with $\delta\subseteq M$, and every $p\in M$, there is an $(M,\mathbb P)$-generic condition $q\leq p$.} Then 
		$\mathbb P$ has the $\omega_1$-approximation property.
	\end{proposition}
	\begin{proof}
		Fix  a $\delta$-preserving forcing $\mathbb P$ and assume that the maximal condition of $\mathbb P$ forces $\dot A$ is a  $\delta$-approximated subset of an ordinal $\gamma$. Pick a regular $\theta$, with $\gamma,\dot{A},\mathcal P(\mathbb P)\in H_\theta$.
		We shall show that $\mathbb P\Vdash ``\dot{A}\in V"$. Let $G\subseteq\mathbb P$ be a $V$-generic filter, and
		set 
		\[
		\mathcal S\coloneqq\{M\in\mathfrak G_\theta: p,\gamma,\dot A,\mathbb P\in M \mbox{ and } M[G]\cap H_\theta^V=M \}.
		\]
		In $V[G]$, $\mathcal S$ is stationary in $\mathcal P_{\omega_2}(H_\theta^V)$. To see this, let $F:\mathcal P_{\omega}(H_\theta^V)\rightarrow \mathcal P_{\omega_2}(H_\theta^V)$ be defined by $F(x)=\{\dot{y}^G\}$ if $x=\{\dot{y}\}$ for some $\mathbb P$-name $\dot{y}$  with $\dot{y^G}\in H_\theta^V$, and otherwise let $F(x)=\{p,\gamma,\dot{A},\mathbb P\}$. The set of models in $\mathfrak G$ which are closed under $F$ is stationary. Observe that
		a model $M\in\mathfrak{G}$ is closed under $F$ if and only if $M\in\mathcal S$.
		
		Let  $A=\dot{A}^G$ and fix $M\in\mathcal S$.
		We claim that $A$ is  $\delta$-approximated in $M$. Let $a\in M$ be a ${<\delta}$-sized subset of $\gamma$. 
		Let $D_a$ be the set of conditions deciding $\dot A \cap a$. Then $D_a$ belongs to $M$ and is dense in $\mathbb P$, as the maximal condition forces that $\dot{A}$ is 
		$\delta$-approximated in $V$. By the elementarity of $M[G]$ in $H_\theta[G]$,
		there is $p\in G\cap D_a\cap M[G]$. But then $p\in M$, as $D_a\in H_\theta^V$.
		Working in $V$,  the elementarity of $M$ in $H_\theta$ implies that there is some $b\in M$ such that, $p\Vdash`` \check b=\dot A\cap a"$. Since $p\in G$, we have 
		$A\cap a=b\in M$. Thus $A$ is  $\delta$-approximated in $M$.  By our assumption, $M$ is an $\delta$-guessing  in $V[G]$. Thus there is $A^*$ in $M$, and hence in $V$,
		such that $A^*\cap M=A\cap M$. 
		
		Working  in $V[G]$ again, for every $M\in\mathcal S$, there is, by the previous paragraph, a set $A^*_M\in M$ such that $A^*_M\cap M=A\cap M$.
		This defines a regressive function $M\mapsto A^*_M$  on $\mathcal S$. As $\mathcal S$ is stationary in $H^V_\theta$, there are a set $A^*\in H^V_\theta$ and
		a stationary set $\mathcal S^*\subseteq\mathcal S$ such that for every $M\in\mathcal S^*$, we have
		$A^*\cap M=A\cap M$. Since $A\subseteq\bigcup\mathcal S^*$, we have $A^* =A$, which in turn implies that $A\in V$.
	\end{proof}
	The following follows from \cref{Lemma0,prop0}.
	\begin{corollary}\label{Lemma1}
		Assume that $\mathbb P$ is an $\delta$-preserving forcing. Suppose that for every sufficiently large regular cardinal $\theta$, $\mathbb P$ is proper for a stationary set $\mathfrak{G}_\theta\subseteq\mathcal P_{\omega_2}(H_\theta)$ of $\omega_1$-gesm of $H_\theta$ with the $\delta$-covering property. Then the following are equivalent.
		\begin{enumerate}
			\item $\mathbb P$ has the $\delta$-approximation property.
			\item  Every $\delta$-guessing model is indestructible by $\mathbb P$.
		\end{enumerate}
	\end{corollary}
	\begin{proof}[\nopunct]
	\end{proof}
	
	In particular, we have the following theorem.
	\begin{corollary}
		Assume that $\mathbb P$ is an $\omega_1$-preserving forcing. Suppose that for every sufficiently large regular cardinal $\theta$, $\mathbb P$ is proper for a stationary set $\mathfrak{G}_\theta\subseteq\mathcal P_{\omega_2}(H_\theta)$ of $\omega_1$-\gesms~of $H_\theta$. Then the following are equivalent.
		\begin{enumerate}
			\item $\mathbb P$ has the $\omega_1$-approximation property.
			\item  Every $\omega_1$-guessing model is indestructible by $\mathbb P$.
		\end{enumerate}
	\end{corollary}
	\begin{proof}[\nopunct]
	\end{proof}

	\begin{corollary}\label{cor1.2}
		Assume $\rm GMP(\omega_2,\omega_1)$. Suppose that $\mathbb P$ is an $\omega_1$-preserving forcing which is also proper for models of size $\omega_1$.  Then the following are equivalent.
		\begin{enumerate}
			\item  $\mathbb P{-\rm IGMP}(\omega_2,\omega_1)$ holds.
			\item $\mathbb P$ has the $\omega_1$-approximation property.
		\end{enumerate}
	\end{corollary}
	\begin{proof}[\nopunct]
	\end{proof}

	The following is a generalisation of  \cref{Special implies SH}.
	
	\begin{theorem}\label{theorem-gen}
		Assume  ${\rm IGMP}(\omega_2,\omega_1)$. Then every $\omega_1$-preserving forcing which is proper for models of size $\omega_1$ has the $\omega_1$-approximation property.
		In particular, under ${\rm IGMP}(\omega_2,\omega_1)$ every $\omega_1$-preserving forcing of size $\omega_1$ has the $\omega_1$-approximation property.
	\end{theorem}
	\begin{proof}
		Let $\mathbb P$ be an $\omega_1$-preserving function which is proper for models of size $\omega_1$. By
		${\rm IGMP}(\omega_2,\omega_1)$,   \cref{Lemma1} implies that $\mathbb P$ has the $\omega_1$-approximation property.
	\end{proof}
	
	Note that in the above proposition, it is enough to assume the nondiagonal version of ${\rm IGMP}(\omega_2,\omega_1)$.
	
	For a class  $\mathfrak K$ of forcing notions and a cardinal $\kappa$, we let 
	${\rm FA}(\mathfrak K,\kappa)$ state that for every $\mathbb P\in\mathfrak K$, and every $\kappa$-sized family $\mathcal D$ of dense subsets of $\mathbb P$, there is a $\mathcal D$-generic filter $G\subseteq\mathbb P$.
	
	\begin{lemma}\label{Lemma2}
		Assume ${\rm FA}(\{\mathbb P\},\kappa)$, for some
		forcing notion $\mathbb P$. Suppose that  $M$ is a $\delta$-guessing set of size $\kappa\geq\delta$.
		Then $\mathbb P$ forces that $M$ is $\delta$-guessing.
	\end{lemma}
	\begin{proof}
		Assume towards a contraction that for some $p_0\in\mathbb P$, some ordinal $\eta\in M$, and some $\mathbb P$-name $\dot{A}$, $p_0$ forces that $\dot{A}\subseteq\eta$  is  $\delta$-approximated in $M$, but is not guessed in $M$. We may assume that $p_0$ is the maximal condition of $\mathbb P$.
		
		\begin{itemize}
			\item For every $\alpha\in M\cap\eta$, let $D_\alpha\coloneqq\{p\in\mathbb P: p  \mbox{ decides } \alpha\in \dot{A}\}$. 
			\item For every $x\in M\cap \mathcal P_{\delta}(\eta)$, let $E_x\coloneqq\{p\in\mathbb P:~  \exists y\in M~  p\Vdash ``\dot{A}\cap x=\check{y}"\}$.
			\item For every $B\in M\cap \mathcal P(\eta)$, let $F_B\coloneqq\{p\in\mathbb P: \exists\xi\in M, (p\Vdash``\xi\in\dot{A}" )\Leftrightarrow \xi\notin B\}$.
		\end{itemize}
		By our assumptions, it is easily seen that the above sets are dense in $\mathbb P$.
		Let \[
		\mathcal D=\{D_\alpha,E_x,F_B: \alpha,x,B \mbox{ as above }\}.
		\]
		We have $|\mathcal D|=\kappa$.
		By ${\rm FA}(\{\mathbb P\},\kappa)$, there is a $\mathcal D$-generic filter $G\subseteq\mathbb P$.
		Let $A^*\subseteq\eta$ be defined by 
		\[
		\alpha\in A^* \mbox{ if and only if }~ \exists p\in G \mbox{ with } p\Vdash ``\alpha\in\dot{A}."
		\]
		By the $\mathcal D$-genericity of $G$, $A^*$ is a well-defined subset of $\eta$ which is  $\delta$-approximated but not guessed in $M$, a contradiction!
	\end{proof}
	
	The following theorem is  immediate from \cref{cor1.2,Lemma2}.
	
	\begin{theorem}
		Let $\mathfrak K$ be a class of forcings which are proper for models of size $\omega_1$.
		Assume that  ${\rm FA}(\mathfrak K,\omega_1)$  and ${\rm GMP}(\omega_2,\omega_1)$ hold. Then, for every forcing $\mathbb P\in\mathfrak K$, $\mathbb P{\rm-IGMP}(\omega_2,\omega_1)$ holds, and $\mathbb P$ has the $\omega_1$-approximation property.
	\end{theorem}
	\begin{proof}[\nopunct]
	\end{proof}

	\subsubsection*{Indestructibility and maximality}

	In his PhD thesis \cite{AvrahamPhD}, Abraham asked if there is a  forcing notion $\mathbb P$ in $\rm ZFC$ such that 
	it does not add new reals, adds a new subset of some ordinal whose initial segments belong to the ground model and that
	the forcing does not collapse any cardinal. Notice that if $\rm CH$ holds, then $\rm Add(\omega_1,1)$ is countably closed and
	$\omega_2$-c.c while adding a new subset of $\omega_1$.
	Recall that also  Foreman's Maximality  Principle (see \cite{FMSH86})
	sates that every nontrivial forcing notion either adds a new real or collapses some cardinals. 
	Using the forcing with initial segments of an uncountable cardinal $\kappa$ ordered with the reverse inclusion,
	one can show that Foreman's Maximality principle violates $\rm GCH$
	and the existence of  inaccessible cardinals.
	\begin{definition}
		For a regular cardinal $\kappa$, the {\em Abraham--Todorčević Maximality Principle} at $\kappa^+$, denoted by $\rm{ ATMP(\kappa^+)}$, states
		that if $2^\kappa<\aleph_{\kappa^+}$, then every forcing which adds a new subset of $\kappa^+$ whose initial segments are in the 
		ground model,  collapses
		some cardinal $\leq 2^{\kappa}$.
	\end{definition} 
	Towards answering the above-mentioned question of Abraham, Todorčević showed in \cite{Todorcevic82} that ${\rm ATMP}(\omega_1)$ is true
	if every tree of size and height $\omega_1$ with at most $\omega_1$ cofinal branches  is weakly special.  This principle was further studied 
	by Golshani and Shelah in \cite{GolShe}, where they showed that ${\rm ATMP}(\kappa^+)$ is consistent for 
	every prescribed regular cardinal $\kappa$.
	Cox and Krueger \cite{CK2017} proved the following.
	
	\begin{proposition}[Cox--Krueger \cite{CK2017}]
		$\rm IGMP(\omega_2,\omega_1)$ implies  $\rm ATMP(\omega_1)$.
	\end{proposition}
	\begin{proof}[\nopunct]
	\end{proof}

	\subsection{${\rm IGMP}^{+}$}

	\begin{theorem}[Mohammadpour--Veličković \cite{RahmanThesis}]
		${\rm IGMP}^+(\omega_3,\omega_1)$ is consistent modulo the consistency of two supercomapct cardinals..
	\end{theorem}
	\begin{proof}[\nopunct]
	\end{proof}
	
	We shall prove that $\rm IGMP^+(\omega_3,\omega_1)$ implies $\rm ATMP(\omega_2)$, and since $\rm IGMP(\omega_2,\omega_1)$ follows from
	$\rm IGMP^+(\omega_3,\omega_1)$, we obtain the consistency of $\rm ATMP(\omega_1)$ and $\rm ATMP(\omega_2)$ simultaneously.

	\begin{theorem}[Mohammadpour--Veličković \cite{RahmanThesis}]\label{nice application}
		Suppose that $V\subseteq W$ are transitive models of $\rm ZFC$. 
		Assume  that $\rm SGM^+(\omega_3,\omega_1)$ 
		and $2^{\omega_1}<\aleph_{\omega_{2}}$ hold in $V$. 
		Suppose that $W$ has a  subset of $\omega^V_2$ which does not belong to $V$.
		Then either $\mathcal P^V(\omega_1)\neq \mathcal P^W(\omega_1)$ or  some $V$-cardinal  $\leq 2^{\omega_1}$ is no longer a cardinal in $W$.
	\end{theorem}
	\begin{proof}
		Let $x\in W\setminus V$ be a  subset of $\omega^V_2$. Assume that $\mathcal P^V(\omega_1)=\mathcal P^W(\omega_1)$.
		We shall show that  some cardinal  $\leq 2^{\omega_1}$ is no longer a cardinal in $W$.
		Since $\mathcal P^V(\omega_1)= \mathcal P^W(\omega_1)$, 
		every initial segment of $x$ belongs to $V$. Letting now $\mathfrak X=\{x\cap\gamma:\gamma<\omega_2\}$, we have that
		$\mathfrak X$ is bounded in $V$.
		Assume towards a contradiction that every $V$-cardinal  $\leq 2^{\omega_1}$ remains  cardinal in $W$.
		Work in $W$, and let $\mu\geq\omega_2$ be the least cardinal
		such that there is a set $M$ in $V$ of cardinality $\mu$ such that $M\cap \mathfrak X$ is of size $\omega_2$. Thus $\mu\leq 2^{\omega_1}$.
		We claim that $\mu=\omega_2$. Suppose that $\mu>\omega_2$ and $M$ is a witness for that,  then one can work in $V$ and write 
		$M$ as the union of an increasing sequence $\langle M_\xi:\xi<{\rm cof^{V}}(\mu)\rangle$ of subsets of $M$ in $V$ whose size are less than $\mu$.
		Since $\mu\leq 2^{\omega_1}<\aleph_{\omega_2}$
		and every cardinal $\leq 2^{\omega_1}$ is a cardinal in $W$, ${\rm cof}^W(\mu)={\rm cof}^V(\mu)\neq\omega_2$. 
		Thus either $\mu$ is of cofinality at most $\omega_1$,
		which then by the pigeonhole principle, there is $\xi<{\rm cof}(\mu)$ such that $|M_\xi\cap\mathfrak X|=\omega_2$, or $\mu$
		is regular, and thus  there is some $\xi<{\rm cof}(\mu)$ such that  $M\cap \mathfrak X\subseteq M_\xi$, but in either case we
		obtain a  contradiction since $|M_\xi|<\mu$.
		Therefore, $\mu=\omega_2$. Let $M$ be a witness for $\mu=\omega_2$, and let $\mathfrak X'=M\cap\mathfrak X$.
		Notice that $V\models |M|=\omega_2$.
		Since $M$ is in $V$ and that $V$ satisfies  $\rm SGM^{+}(\omega_3,\omega_1)$, one can
		cover $M$ with an indestructibly strongly $\omega_1$-guessing model $N$ of size $\omega_2$.
		Working in $W$, 
		$x$ is countably approximated in $N$, since if $\gamma\in N\cap \omega_2$, then there is
		$\gamma'>\gamma$ in $N$ such that $x\cap\gamma'\in\mathfrak X'\subseteq N$, and hence $x\cap\gamma\in N$.
		On the other hand
		$N$ is a guessing model in $W$ by $\rm SGM^+(\omega_3,\omega_1)$ in $V$ and that  both $\omega_1$ and $\omega_2$ are cardinals in $W$.
		Thus $x$ is guessed in $N$, but then $x$ must be in $N$ as $|x|\leq |N|$.
		Therefore, $x$ is in $V$, which is a contradiction!
	\end{proof} 
	
	The following corollaries are immediate.
	\begin{corollary}
		Suppose that $V\subseteq W$ are transitive models of $\rm ZFC$. Assume in $V$, $\rm SGM^+(\omega_3,\omega_1)$ 
		, $2^{\aleph_0}<\aleph_{\omega_{1}}$ and $2^{\aleph_1}<\aleph_{\omega_{2}}$ hold.
		Suppose that $W$ has a new subset of $\omega^V_2$.
		Then either $\mathbb R^{W}\neq\mathbb R^V$ or some $V$-cardinal $\leq 2^{\omega_1}$  is collapsed in $W$. 
	\end{corollary}
	\begin{proof}[\nopunct]
	\end{proof}
	\begin{corollary}
		$\rm SGM^+(\omega_3,\omega_1)$ implies ${\rm ATMP}(\omega_1)$ and ${\rm ATMP}(\omega_2)$.
	\end{corollary}
	\begin{proof}[\nopunct]
	\end{proof}

	\[
	\adjustbox{scale=1.1,center}{
		\begin{tikzcd}
			\neg {\rm CH}&  & &  & &2^{\aleph_0}\geq \aleph_3  \\
			\neg w{\rm KH} \arrow[u]&   \neg \square(\omega_2,\lambda)& \rm TP(\omega_2) & \rm SCH &\neg {\rm ATMP}(\omega_1) & {\rm MP}(\omega_2)\arrow[u] \\
			& & & &  \\
			& & & &\\
			& &  & &\\
			&  &  {\rm GMP}(\omega_2,\omega_1)\arrow[uuuu]\arrow{uuuul}[sloped]{ Weis} \arrow{uuuur}[sloped]{ Krueger+Viale}\arrow{uuuurr}[sloped]{ Viale-Weis}\arrow{uuuull}[sloped]{ Cox-Krueger}\arrow[uuuurrr, red, "\textcolor{red}{ \textbf{/}}" marking] \arrow[red,  "\textbf{/}" marking]{drr}[sloped, near start]{  \textcolor{black}{Cox-Krueger}} \arrow[red,  "\textbf{/}" marking]{dll}[sloped, near start]{  \textcolor{black}{Cox-Krueger}} &  & \\
			{\rm ATMP}(\omega_1)\arrow[rrrr]& & &   &{\rm SH}\\
			&   &{\rm IGMP}(\omega_2,\omega_1) \arrow[uu, crossing over] \arrow{ull}[sloped]{\textcolor{black}{Cox-Krueger}} \arrow{urr}[sloped]{ \textcolor{black}{Cox-Krueger}}& &\\
			& & & &  \\
			& & & &  \\
			& & & &  \\
			& &  {\rm PFA} \arrow{uuuu}[sloped]{Cox-Krueger}   & & 
	\end{tikzcd}}
	\]
	
	\subsection{Open problems}

	\begin{problem}
		It is consistent to have $\rm IGMP^+(\omega_{n+2},\omega_1)$ , for every $n\geq 1$. 
	\end{problem}
	
	\begin{problem}
		Is ${\rm IGMP}^{\infty}$ consistent?
	\end{problem}

	\begin{problem}
		Assume $\rm PFA$. For which forcing notions $\mathbb P$-special $\omega_1$-guessing models of size $\omega_1$ form a stationary set.
	\end{problem}
	
	\begin{problem}
		Is ${\rm IGMP}(\omega_2,\omega_1)$ consistent with ${\rm MA}_{\omega_1}+2^{\aleph_0}>\aleph_2$? 
	\end{problem}

	\[
	\adjustbox{scale=1.0,center}
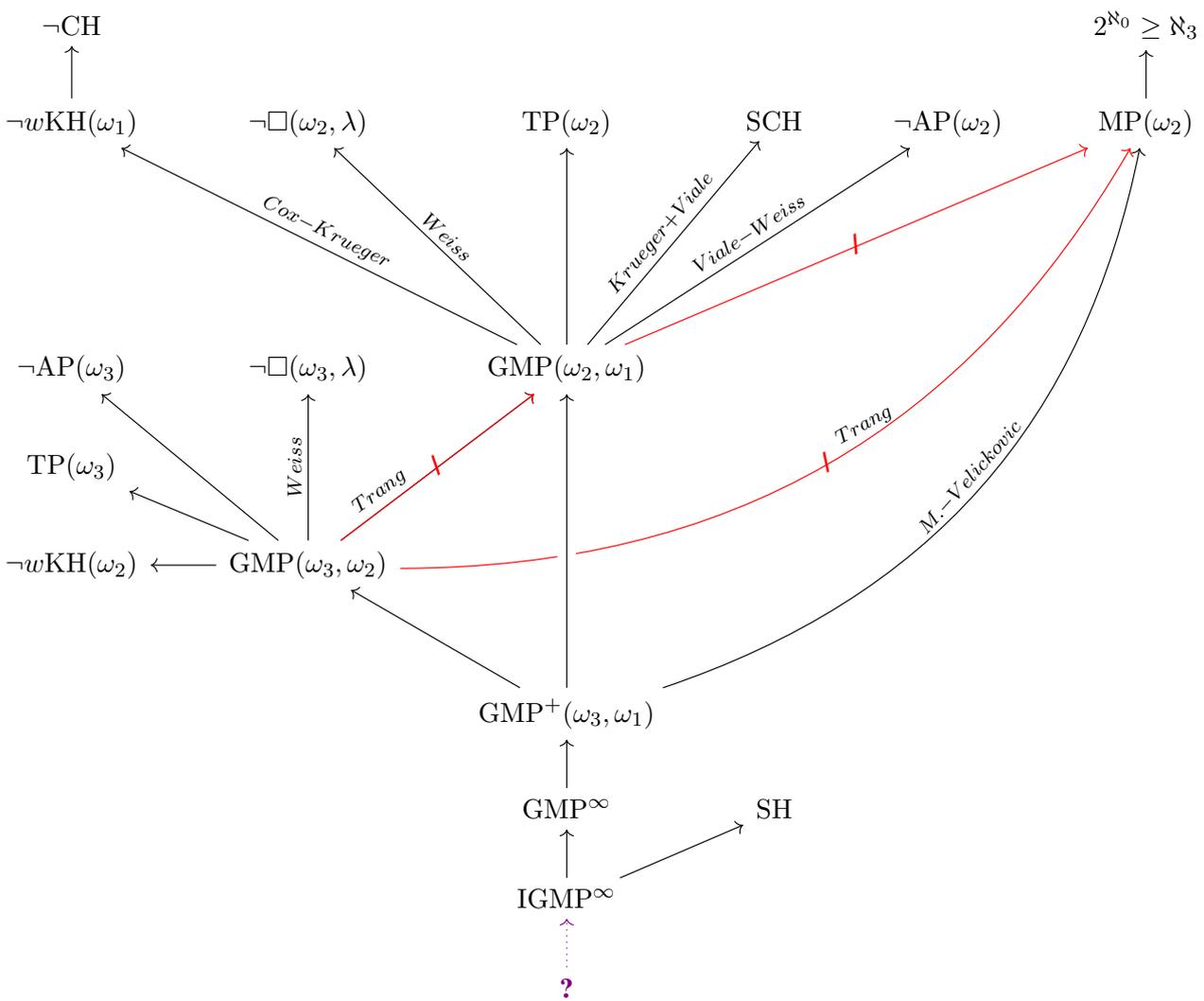
{
		\begin{tikzcd}
			\neg \rm CH&  & &  & &2^{\aleph_0}\geq \aleph_3  \\
			\neg w{\rm KH}(\omega_1) \arrow[u]&   \neg \square(\omega_2,\lambda)& \rm TP(\omega_2) & \rm SCH &\neg {\rm AP}(\omega_2) &
			\rm MP(\omega_2)\arrow[u] \\
			& & & &  \\
			& & & &\\
			& &  & &\\
			\neg{\rm AP}(\omega_3)&  \neg \square(\omega_3,\lambda) &  \rm GMP(\omega_2,\omega_1)\arrow[uuuu]\arrow{uuuul}[sloped]{ Weiss} \arrow{uuuur}[sloped]{ Krueger+Viale}\arrow{uuuurr}[sloped]{ Viale-Weiss}\arrow{uuuull}[sloped]{ Cox-Krueger}\arrow[uuuurrr, red, "\textcolor{red}{ \textbf{/}}" marking]& & \\
			{\rm TP}(\omega_3)& & &\\
			\neg w{\rm KH}(\omega_2) & {\rm GMP}(\omega_3,\omega_2)
			\arrow{l}\arrow{ul}
			\arrow{uur}[sloped, near start]{Trang}
			\arrow[uur, red, "\textcolor{red}{ \textbf{/}}" marking]
			%	\arrow[bend right, red]{uuuuuurrrr}[near start]{Trang}
			\arrow[ bend right, red, "\textcolor{red}{ \textbf{/}}" marking]{uuuuuurrrr}
			[sloped,black]{~~~~~~~~~~~~~~~Trang} 
			\arrow{uu}[sloped]{Weiss}
			\arrow{uul}
			&&  &  &\\
			& & & &   \\
			& &  {\rm GMP}^+(\omega_3,\omega_1) \arrow[uuuu,crossing over] 
			\arrow[uul]  \arrow[ bend right]{uuuuuuuurrr}[sloped]{M.-Velickovic}   & & \\
			& &{\rm GMP}^{\infty}\arrow[u] & {\rm SH}&  \\
			%& & & {\rm SH} & \\
			& &{\rm IGMP^{\infty}}\arrow[u]\arrow[ur] &  & \\
			& &\textcolor{violet}{\textbf{?}} \arrow[violet,dotted]{u}&  & 
		\end{tikzcd}
	}
	\]
	
	\clearpage
	\[
	\adjustbox{scale=1,center}
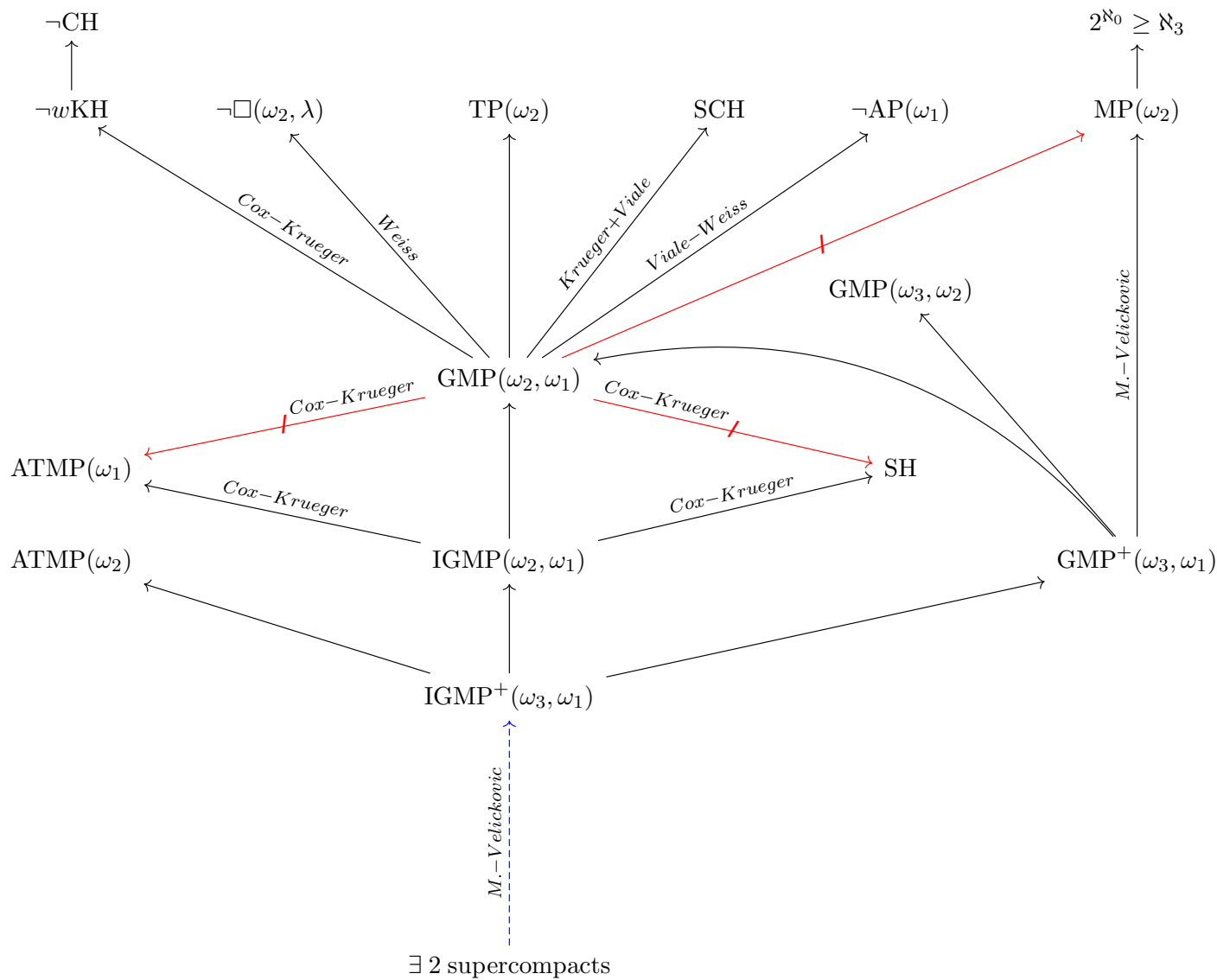
{
		\begin{tikzcd}
			\neg {\rm CH}&  & &  & &2^{\aleph_0}\geq \aleph_3  \\
			\neg w{\rm KH} \arrow[u]&   \neg \square(\omega_2,\lambda)& \rm TP(\omega_2) & \rm SCH &\neg {\rm AP}(\omega_1) & \rm MP(\omega_2)\arrow[u] \\
			& & & & &  \\
			& & & &&\\
			& &  & &{\rm GMP}(\omega_3,\omega_2)&\\
			&  &  {\rm GMP}(\omega_2,\omega_1)\arrow[uuuu]\arrow{uuuul}[sloped]{ Weiss } \arrow{uuuur}[sloped]{ Krueger+Viale}\arrow{uuuurr}[sloped]{Viale-Weiss }\arrow{uuuull}[sloped]{\textcolor{black}{Cox-Krueger}}
			\arrow[uuuurrr, red, " \textbf{/}" marking] \arrow[red, "\textbf{/}" marking]{drr}[sloped, near start]{  \textcolor{black}{Cox-Krueger}} \arrow[red, "\textbf{/}" marking]{dll}[sloped, near start]{ \textcolor{black}{Cox-Krueger}} &  & &\\
			{\rm ATMP}(\omega_1)& & &   &{\rm SH}&  \\
			{\rm ATMP}(\omega_2)&   &{\rm IGMP}(\omega_2,\omega_1) \arrow{uu} \arrow{ull}[sloped]{Cox-Krueger} \arrow{urr}[sloped]{Cox-Krueger}&  & &{\rm GMP}^+(\omega_3,\omega_1)
			\arrow{uuuuuu}[sloped]{M.-Velickovic} \arrow[uulll, bend right]\arrow[uuul]\\
			&&&&& \\
			& &  {\rm IGMP}^+(\omega_3,\omega_1) \arrow[uu] \arrow[uull] \arrow[uurrr]  & & & \\
			&&&&& \\
			&&&&& \\
			&&&&& \\
			&&&&& \\
			& &{\exists~{\rm 2~supercompacts} }\arrow[blue,dashed]{uuuuu}[sloped]{\textcolor{black}{M.-Velickovic}} & & &
		\end{tikzcd}
	}
	\]
	\section{wGMP}\label{secwGMP}
	
	This last section of this survey is devoted to an interesting weakening of  guessing property introduced by Cox and Krueger \cite{CK2017}.   All the following results are due to Cox and Krueger.
	
	\begin{definition}
		Suppose $\kappa$ is a regular uncountable cardinal.
		A set $M$ has the \textbf{weak $\bm \kappa$-guessing property} if for every function $f:\kappa\rightarrow {\rm ORD}$ such that $f\restriction \alpha\in M$, for every $\alpha\in M\cap \kappa$, there is $f^*\in M$ with $f^*\restriction M=f\restriction M$.
	\end{definition}
	
	\begin{definition}
		A set $M$ of size $\omega_1$ is called \textbf{weakly guessing} if for every uncountable regular $\kappa\in M$, $M$ has the weak
		$\kappa$-guessing property.
	\end{definition}
	
	Similarly, one can define the weak approximation property.
	It must be now easy to state the weak-guessing model principles, denoting as before following a \emph{w}.
	Interestingly, many  consequences of guessing model principles follow from this, however the structure of the models are not quite the same as before.
	Similarly, \emph{IwGMP} denotes the indestructible version of the weak guessing model principle.
	
	\begin{theorem}
		\begin{enumerate}
			\item ${\rm wGMP}$ implies $\neg w{\rm KP}$, ${\rm TP}(\omega_2)$, $\neg \square_\kappa$ for all $\kappa\geq\omega_1$,  $\neg \square(\lambda)$ for all $\lambda\geq\omega_2$,  and ${\rm AP}(\omega_1)$ hold.
			\item 	 ${\rm IwGMP}$ implies ${\rm ATMP(\omega_2)}$.
		\end{enumerate}
	\end{theorem}

	\begin{proposition}
		The set of weakly $\omega_1$-guessing submodels of $H_\theta$ is closed under countable union of $\subseteq$-increasing sequences.
	\end{proposition}
	
	\begin{theorem}
		${\rm GMP}(\omega_2,\omega_1,H_{\omega_2})$ is equivalent to ${\rm wGMP}(\omega_2,\omega_1,H_{\omega_2})$
	\end{theorem}
	The main theorem of Cox and Kruger reads as follows.
	\begin{theorem}
		Assume the existence of a supercompact cardinal with infinitely many measurable cardinals above it. Then in a generic extension,  there exist stationarily many 
		$N\in \mathcal P_{\omega_2}(H_{\aleph_{\omega+1}})$
		such that $N$ is indestructibly weakly guessing, has uniform cofinality $\omega_1$, and is not internally unbounded.
	\end{theorem}
	
	Recall Krueger's theorem that most $\omega_1$-guessing models are internally unbounded.
	\subsection{Open Problems}
	\begin{problem}[Cox--Krueger]
		Does the existence of stationarily many indestructibly weakly guessing models which are not internally unbounded follow from Martin’s maximum?
	\end{problem}
	\begin{problem}[Cox--Krueger]
		Does $\rm wGMP$ imply $\rm GMP$, or $\rm wIGMP$ imply $\rm IGMP$?
	\end{problem}
	
	%%%%%%%%%%%%%%%%%%%%%%%%% Acknowledgment and Bibliography %%%%%%%%%%%%%
	%%%%%%%%%%%%%%%%%%%%%%%%%%%%%%%%%%%%%%%%%%%%%%%%%%%%%%%%%%%%%%%%%%%%%%%
	%%%%%%%%%%%%%%%%%%%%%%%%%%%%%%%%%%%%%%%%%%%%%%%%%%%%%%%%%%%%%%%%%%%%%%%%

	\bigskip
	\footnotesize
	\noindent\textit{Acknowledgements.}
	The author thanks the Austrian Science Fund (FWF) for their support through the Lise Meitner project M 3024 (P.I. the author) and  Elise Richter Project V844 (P.I. Sandra Müller).

	\normalsize
	\baselineskip=17pt

	\printbibliography[heading=bibintoc]

\end{document}